\documentclass[reqno,12pt]{amsart}
\usepackage[
  colorlinks,
  hypertexnames=false
]{hyperref}

\usepackage{xurl}
\usepackage{fancybox,fancyhdr,graphics,epsfig}
\usepackage{subcaption}

\usepackage{amsmath,amssymb,mathrsfs,mathtools,verbatim} 
\usepackage{ae,aecompl}
\usepackage[margin=1in]{geometry}
\usepackage{microtype}
\usepackage{tikz}
\usepackage{dsfont}
\usepackage{xcolor}
\definecolor{teal}{rgb}{0,0.45,0.45}
\usepackage{amsthm}

\usepackage[
  backend=biber,
  style=alphabetic,
  maxbibnames=99,
  maxcitenames=99,
  maxalphanames=5,
  minalphanames=5,
  backref=true,
  doi=true,
  eprint=true,
  isbn=false,
  url=false
]{biblatex}

\AtEveryBibitem{\clearfield{issn}}

\DeclareDelimFormat{multicitedelim}{\addcomma\allowbreak}
\DeclareFieldFormat[article]{journaltitle}{#1}
\DeclareFieldFormat[article]{title}{\mkbibemph{#1}}% Article titles: italic, no quotation marks
\DeclareFieldFormat[article]{volume}{\mkbibbold{#1}}% Journal volume: bold
\renewbibmacro{in:}{%
  \ifentrytype{article}
    {}
    {\printtext{\bibstring{in}\intitlepunct}}}
\renewbibmacro*{journal+issuetitle}{%
  \usebibmacro{journal}%
  \setunit*{\addspace}%
  \iffieldundef{series}
    {}
    {\newunit
     \printfield{series}%
     \setunit{\addspace}}%
  \printfield{volume}%
  \setunit{\addspace}%
  \usebibmacro{issue+date}%
  \setunit{\addcomma\space}%
  \iffieldundef{number}
    {}
    {\printtext{no\adddot\space}%
     \printfield{number}}%
  \newunit}
\DeclareFieldFormat{doi}{%
  \href{https://doi.org/#1}{doi:\nolinkurl{#1}}}
\DeclareFieldFormat{eprint:arxiv}{%
  \href{https://arxiv.org/abs/#1}{arXiv:\nolinkurl{#1}}}
\renewbibmacro*{pageref}{%
  \iflistundef{pageref}
    {}
    {\setunit{\addspace}%
     \printtext{\ensuremath{\uparrow}\printlist{pageref}}}}

\newtheorem{thm}{Theorem}
\newtheorem{prop}[thm]{Proposition}
\newtheorem{lem}[thm]{Lemma}
\newtheorem{cor}[thm]{Corollary}
\newtheorem*{cor*}{Corollary}
\newtheorem*{thm*}{Theorem}
\newtheorem{exam}[thm]{Example}
\newtheorem{remk}[thm]{Remark}
\newtheorem{defn}[thm]{Definition}

\newtheorem{theorem}[thm]{Theorem}
\newtheorem{proposition}[thm]{Proposition}
\newtheorem{lemma}[thm]{Lemma}
\newtheorem{corollary}[thm]{Corollary}
\newtheorem{remark}[thm]{Remark}
\newtheorem{definition}[thm]{Definition}
\newtheorem{conjecture}[thm]{Conjecture}
\newtheorem{example}[thm]{Example}

\newcommand{\bbT}{\mathbb{T}}
\newcommand{\bbC}{\ComplexNumbers}

\newcommand{\bbF}{\mathbb{F}}
\newcommand{\bbZ}{\mathbb{Z}}

\newcommand{\mL}{{\mathcal L}}
\newcommand{\mA}{{\bf{A}}}

\newcommand{\mE}{{\mathcal E}}
\newcommand{\mC}{{\mathcal C}}

\newcommand{\mO}{{\mathcal O}}

\newcommand{\proj}{{\bf p}}

\newcommand{\ii}{{\bf i}}

\newcommand{\cp}{{\mathbb{CP}}}

\hypersetup{
	colorlinks,
	linkcolor={red!45!black},
	citecolor={blue!45!black},
	urlcolor={black}
}

\usepackage[english]{babel}

\usepackage{bookmark}

\numberwithin{equation}{section}

\newcommand{\Span}{\mathrm{Span}}
\newcommand{\Reals}{\mathbb{R}}
\newcommand{\ComplexNumbers}{\mathbb{C}}

\newcommand{\mm}{\mathfrak{m}}
\newcommand{\Vol}{\mathrm{Vol}}
\newcommand{\Area}{\mathrm{Area}}
\newcommand{\Homology}{\mathrm{H}}

\newcommand{\sumcyclic}{\sum_{\circlearrowright\,i,j,k}}
\newcommand{\surface}{S}
\newcommand{\citep}[2]{\cite[#1]{#2}} %necessary when in the [] argument of a theorem for instance
\newcommand{\blowup}{\mathrm{Bl}}

\newcommand{\ttt}{\ft}

\newcommand{\TT}{\xi} %basis of u(2)
\newcommand{\Zodd}{\bbZ_{\mathrm{odd}}}
\newcommand{\Zeven}{\bbZ_{\mathrm{even}}}
\newcommand{\AAA}{\cA}
\newcommand{\BBB}{\cB}
\newcommand{\CCC}{\cC}

\newcommand{\bC}{{\bf C}}

\newcommand{\cA}{\mathcal{A}}
\newcommand{\cB}{\mathcal{B}}
\newcommand{\cC}{\mathcal{C}}

\newcommand{\cL}{\mathcal{L}}

\newcommand{\fg}{{\mathfrak g}}

\newcommand{\fh}{{\mathfrak h}}

\newcommand{\fk}{{\mathfrak k}}

\newcommand{\ft}{{\mathfrak t}}

\newcommand{\C}{\bC}

\newcommand{\su}{\mathfrak{su}}

\newcommand{\SU}{{\rm SU}}

\newcommand{\U}{{\rm U}}

\newcommand{\Ad}{\mathrm{Ad}}

\newcommand{\End}{{\mathrm{End}}}

\renewcommand{\epsilon}{\varepsilon}

\newcommand{\Hom}{{\mathrm{Hom}}}

\newcommand{\ad}{\mathrm{ad}}

\newcommand{\del}{\partial}
\newcommand{\diag}{\mathrm{diag}}

\newcommand{\id}{\mathrm{id}}

\renewcommand{\Im}{\mathop{\mathrm{Im}}}

\newcommand{\rank}{\mathop{\mathrm{rank}}}

\renewcommand{\Re}{\mathop{\mathrm{Re}}}

\newcommand{\tr}{\mathop{\mathrm{tr}}\nolimits}

\newcommand{\vol}{\mathrm{vol}}

\newcommand{\sign}{\mathrm{sign}}

\def\<{\mathopen{}\left<}
\def\>{\right>\mathclose{}}
\def\({\mathopen{}\left(}
\def\){\right)\mathclose{}}

\newcommand{\Comment}[2][\empty]{\ifthenelse{\equal{#1}{\empty}}{\todo[color=gray!10]{#2}}{\todo[color=gray!10,#1]{#2}}}%inline comment
\allowdisplaybreaks

\title{Deformed Hermitian--Yang--Mills equation on the manifold of full flags}
\date{\today}

\author{Benoit Charbonneau}
\address[Benoit Charbonneau]{Department of Pure Mathematics and Department of Physics and Astronomy, University of Waterloo, Ontario, Canada}
\urladdr{\href{https://www.math.uwaterloo.ca/~bcharbon/}{www.math.uwaterloo.ca/~bcharbon}}
\email{\href{mailto:benoit@alum.mit.edu}{benoit@alum.mit.edu}}
\author{Gon\c{c}alo Oliveira}
\address[Gon\c{c}alo Oliveira]{Centro de An\'alise Matem\'atica, Geometria e Sistemas Din\^amicos, Departamento de Matem\'atica,
Instituto Superior T\'ecnico, Lisboa, Portugal}
\urladdr{\href{https://sites.google.com/view/goncalo-oliveira-math-webpage}{sites.google.com/view/goncalo-oliveira-math-webpage}}
\email{\href{mailto:goncalo.m.f.oliveira@tecnico.ulisboa.pt}{goncalo.m.f.oliveira@tecnico.ulisboa.pt}}
\author{Rosa Sena-Dias}
\address[Rosa Sena-Dias]{Centro de An\'alise Matem\'atica, Geometria e Sistemas Din\^amicos, Departamento de Matem\'atica,
Instituto Superior T\'ecnico, Lisboa, Portugal}
\urladdr{\href{https://www.math.tecnico.ulisboa.pt/~senadias}{www.math.tecnico.ulisboa.pt/~senadias}}
\email{\href{mailto:rsenadias@math.ist.utl.pt}{rsenadias@math.ist.utl.pt}}


\begin{document}
\maketitle

%===============================================================================
\begin{abstract}
	We construct the first example of a higher rank, irreducible deformed Hermitian--Yang--Mills (dHYM) connection in the small radius regime.
	
	We also construct these in the large radius regime on infinitely many different bundles and make some contributions to the rank one equation as well. In particular, we investigate solutions away from the supercritical regime, showing the existence of solutions with any possible angle, and rule out some possible stability conditions.
\end{abstract}
%===============================================================================

%===============================================================================

\tableofcontents

\section{Introduction}

\subsection{The rank one equation}

Let $(X,\omega)$ be a connected, compact K\"ahler manifold of complex dimension $n$ and $\Omega\in \Homology^{1,1}(X,\Reals)$. We say that $\kappa \in \Omega$ solves the \emph{deformed Hermitian--Yang--Mills (dHYM)} equation if there is a real constant $\theta$ such that
\begin{equation}\label{eqn:dHYM}
	\Im \left( e^{-\ii \theta } (\omega+\ii\kappa)^n \right)=0.
\end{equation}

Define
\begin{equation}
	Z_X(\Omega):= \int_X \frac{(\omega+\ii\kappa)^n}{n!}.
\end{equation} 
Suppose that $Z_X(\Omega)\ne 0$. Integrating Equation \eqref{eqn:dHYM}, we see that $\theta$ is in fact  the argument of $Z_X(\Omega)$ and depends only, mod $\pi$, on the classes $[\omega]$ and $\Omega$. 

Locally, we can always find real valued functions $\lambda_1,\ldots,\lambda_n$ describing the eigenvalues of $\kappa\omega^{-1}$.
A short computation shows that Equation \eqref{eqn:dHYM} can be then written as
\begin{equation}\label{eqn:dHYM 2}
	\Theta_\omega (\kappa):= \sum_{i=1}^n \arctan(\lambda_i) \equiv \theta \mod \pi.
\end{equation}
{One of the first difficulties that arises in studying the dHYM equation is precisely that the constant $\theta$, when defined, involves a choice that should be made coherently in $\Homology^{1,1}(X,\Reals)$. In certain circumstances, Collins and Yau (see \cite{Collins-Yau-MomentMaps}) show that it is possible to make such a choice, canonically. 
Let
\begin{equation*}
\mathcal{H}_{\Omega}:=\{\varphi\in \mC^\infty(X) :  \Re \left( e^{-\ii \theta } (\omega+\ii(\kappa+\ii\del\bar \del \varphi))^n \right)>0\}.
\end{equation*}
\begin{defn}
Let $(X^n,\omega)$ be a connected, compact K\"ahler manifold and $\Omega\in \Homology^{1,1}(X,\Reals)$. Assume that $Z_X(\Omega)\ne 0$ and $\mathcal{H}_{\Omega}$ is non-empty. Then, there is a unique $\hat\theta$ such that
\begin{enumerate}
\item The set $\mathcal{H}_{\Omega}$ is given by
\[
\mathcal{H}_{\Omega}=\{\varphi\in \mC^\infty(X) :  |\Theta_\omega(\kappa+\ii\del\bar \del \varphi)-\hat\theta|<\pi/2\}.
\]
\item The constant $\hat\theta$ equals the argument of $Z_X(\Omega)$ modulo $2\pi$.
\end{enumerate}
Such a constant is called the \emph{(analytic) lifted angle} and takes values in $(-n\frac{\pi}{2} , n\frac{\pi}{2}) \subset \Reals$.
\end{defn}
See also \cite{CollinsXieYau-dHYM-geometry-physics} for a related discussion. Under the same assumptions, it is shown in \cite{CollinsXieYau-dHYM-geometry-physics} that the lift $\hat\theta$ only depends on the class $\Omega$ for a fixed K\"ahler manifold. Using this lifted angle, one can define the notion of \emph{supercritical} or \emph{hypercritical} classes in $\Homology^{1,1}(X, \Reals)$, 
\begin{defn}
Let $(X^n,\omega)$ be a connected, compact K\"ahler manifold and $\Omega\in \Homology^{1,1}(X,\Reals)$. Assume that $Z_X(\Omega)\ne 0$ and $\mathcal{H}_{\Omega}$ is non-empty. Then, we say that $\Omega$ is \emph{supercritical} if the lifted angle of $\Omega$  takes values in $(n\frac{\pi}{2}-\pi , n\frac{\pi}{2})$. We say that $\Omega$ is \emph{hypercritical} if the lifted angle takes values in $(n\frac{\pi}{2}-\frac{\pi}{2} , n\frac{\pi}{2})$.
\end{defn}}

The dHYM equation was first introduced as the ``mirror" of the special Lagrange equation by Leung--Yau--Zaslow in \cite{Leung-Yau-Zaslow-AdvThMathPhys}, and Mari\~no--Minasian--Moore--Strominger in \cite{Nonlinear-instantons-SS-p-branes} independently. It has been the object of intense investigation in recent years; see for instance
\cite{
Jacob-Yau-specialLag-linebundle,
Collins-Jacob-Yau-forms-with-specified-Lagrangian-phase,
Chen-J-eqn-supercritical-dHYM,
Jacob-Sheu-dHYM-blowup-Pn,
Ballal-supercritical-dHYM-projective,
Collins-Shi-stability-dHYM,
Collins-Lo-Shi-Yau-stability-line-bundles-dHYM-EllipticSurfaces,
Han-Jin-rigidity-dHYM,
Han-Jin-Chern-ineq-dHYM-4d,
Huang-Zhang-Cauchy-Dirichlet-dHYM,
Pingali-note-dHYM-PDE,
Pingali-dHYM-3folds,
Sun-boundary-case-supercritical-dHYM,
Jacob-dHYM-level-sets,
Jacob-weakgeodesics-dHYM,
Lin-dHYM-Positivstellensatz,
Lin-dHYM-compact-Hermitian-mflds,
Schlitzer-Stoppa-dHYM-extended-gauge-group,
Sheu-thesis,
McCarthy-Benjamin-thesis,
Fan-notes-dHYM-large-limits,
Yamamoto-specialLag-dHYM-tropical,
Fu-Yau-Zhang-critical-LYZ-in-Kahler,
Murakami-Jeqn-dHYM-holomorphic-submersions}
 and others cited elsewhere in this paper. The equation can also be understood as an interpolation between the Hermitian--Yang--Mills equation, 
and the J-equation introduced by Donaldson in \cite{Donaldson-moment-maps-diffeomorphisms} and studied initially by Chen, Weinkove, Song, Lejmi and Szekelyhidi (\cite{Chen-parabolicflow-Kahler,Weinkove-convergence-Jflow-Kahler,Weinkove-Jflow-higher-dim-lower-boundedness-Mabuchi,SongWeinkove-convergence-J-flow-Mabuchi-energy,LejmiSzekelyhidi-Jflow-stability}). Both equations are paradigmatic in K\"ahler geometry as they exhibit the typical behaviour we see in many PDEs of interest where existence of solutions is governed by a stability condition. It is often an important step in studying such PDEs to discover what the stability condition is. This discovery can sometimes be done by relating the PDE to a moment map setup involving infinite dimensional group actions. In fact, the similarities between dHYM and the J-equation go beyond this general set up (see \cite{Chen-J-eqn-supercritical-dHYM}). Both arise as complex Hessian equations (see \cite{Fang-Ma-on-fully-nonlinear-elliptic}); however, dHYM carries an additional ambiguity associated with the phase, analogous to that of the special Lagrangian equation. For analytic background on complex Hessian equations see \cite{Szekelyhidi-fully-non-linear-elliptic,PhongPicardZhang-2ndorder-estimate-general-cplx-Hessian,HouMaWu-2nd-order-estimate-cplxHessian-eqn}. In the two cases, stability has been formulated numerically (see \cite{Chen-J-eqn-supercritical-dHYM,Datar-Pingali-numerical-criterion-Monge-Ampere,Song-Nakai-Moishezon-cplx-Hessian-eqn}) and the PDE's have been studied through associated flows. The J-flow was introduced in \cite{Chen-parabolicflow-Kahler} and has been extensively studied (see \cite{Weinkove-convergence-Jflow-Kahler,Weinkove-Jflow-higher-dim-lower-boundedness-Mabuchi,SongWeinkove-convergence-J-flow-Mabuchi-energy,Collins-Szekelyhidi-Jflow-toric}). There are several flows related to dHYM equations that have been introduced more recently in \cite{Jacob-Yau-specialLag-linebundle} and also \cite{Takahashi-TanConcavity,Fu-Yau-Zhang-new-flow-LYZ-eqn} and that are not as well understood. Similarities between the J-flow and Fu--Yau--Zhang's flow for instance become apparent in Murakami's recent work \cite{Murakami-weak-limits-Jflow-dHYM-on-Kahler-surfaces}.

In this paper we focus on the dHYM equation and numerical criteria for existence of solutions. In fact, we look at a specific manifold where dHYM can be studied explicitly.

In \cite{Collins-Jacob-Yau-forms-with-specified-Lagrangian-phase}, Collins--Jacob--Yau conjectured that, in the supercritical regime, the existence problem for the dHYM connection is governed by an algebraic stability condition which we state. We start with a definition.
\begin{defn}\label{def:Z}
Let $(X^n,\omega)$ be a connected, compact K\"ahler manifold and let $V$ be a subvariety in $X$. The \emph{central charge} of $V$ is a function $\Homology^{1,1}(X)\to \ComplexNumbers$ defined to be 
\begin{equation}
	Z_V (\Omega):=-\int_V e^{-\ii(\omega + \ii\kappa)}
\end{equation}
for any $\Omega\in \Homology^{1,1}(X,\Reals)$, and representative $\kappa\in\Omega$.
For simplicity, we can write $Z_V(\kappa)$ and $Z_V(\Omega)$ interchangeably. 
\end{defn}	
The conjecture can be stated as follows.
\begin{conjecture}[Collins--Jacob--Yau]\label{cjy_conjecture}
Let $(X^n,\omega)$ be a connected, compact K\"ahler manifold and $\Omega\in \Homology^{1,1}(X,\Reals)$. Assume that $Z_X(\Omega)\ne 0$, $\mathcal{H}_{\Omega}$ is non-empty, and $\Omega$ is supercritical. Then, the following are equivalent 
\begin{enumerate} 
\item The dHYM equation admits a solution\label{exitsdHYMsol}.
\item For every proper subvariety $V\subset X$,\label{central>0}
\begin{equation}
		\Im \left( \frac{Z_V(\Omega)}{Z_X(\Omega)} \right)>0.
	\end{equation}		
\end{enumerate}
\end{conjecture}
We now know that (\ref{exitsdHYMsol}) can only imply (\ref{central>0}) when $\Omega$ is supercritical (see \cite[Remark (1.10)]{Chen-J-eqn-supercritical-dHYM}). Though the supercritical hypothesis was not explicitly written in \cite{Collins-Jacob-Yau-forms-with-specified-Lagrangian-phase},  it was clear from the context that the conjecture was formulated with that regime in mind. In that paper, Collins--Jacob--Yau proved this conjecture among K\"ahler surfaces. In a particularly remarkable breakthrough, Chen \cite{Chen-J-eqn-supercritical-dHYM} and Chu--Lee--Takahashi \cite{Chu-Lee-Takahashi-NakaiMoishezon-dHYM}, proved the supercritical version of the Collins--Jacob--Yau conjecture for {projective manifolds}. {Conjecture \ref{cjy_conjecture} was further refined in the 3 dimensional case in \cite{Collins-Yau-MomentMaps-arxiv} (see Conjectures 8.5 and 8.7) for the supercritical regime. But we know far less about existence of solutions outside the supercritical regime; in fact, in \cite{Zhang-note-supecritical-dHYM}, Zhang constructs examples where the numerical condition holds but no supercritical dHYM solution exists.

Collins--Jacob--Yau also formulate the following conjecture characterising hypercritical classes (see also \cite{Chu-Lee-hypercritical-dHYM-revisited} for a discussion).
\begin{conjecture}[\citep{conjecture 8.7}{Collins-Yau-MomentMaps-arxiv}]\label{conj8.7CY} The set
$\mathcal{H}_{\Omega}$ is non-empty and $\Omega$ has hypercritical phase if and only if for any irreducible subvariety $V$ of $X$, $\Im(Z_V(\Omega)) > 0$.
\end{conjecture}}

In this work, we are able to write down explicit examples of deformed Hermitian--Yang--Mills classes with any possible angle $\hat\theta$. We do so by using symmetry techniques to explicitly solve the equations in the complex $3$ dimensional manifold $\bbF_2$ of full flags in $\bbC^3$. 

\begin{theorem}[see Theorem \ref{thm:Theta_surjective}]\label{thm:Surjective_INTRO}
	Let $\omega$ be a homogeneous K\"ahler metric on $\bbF_2$. Then, for all $\hat{\theta} \in \left(-\frac{3\pi}{2} , \frac{3\pi}{2}\right)$ there is a solution $\kappa$ to the dHYM equation with 
$\Theta_{\omega}(\kappa) = \hat{\theta}$. 
\end{theorem}

We use the construction from which the above theorem arises to show that the supercritical condition is essential in Conjecture \ref{cjy_conjecture},  although we show something different from Zhang. More precisely, we show that without supercriticality, the existence of a solution to the dHYM equation does not imply the numerical criteria appearing in Conjecture \ref{cjy_conjecture} are satisfied whereas Zhang shows that the numerical criteria alone do not imply existence and supercriticality. {We are also able to give counterexamples to Conjecture \ref{conj8.7CY}. Note that when there is a solution to the dHYM equation in a given class, the space $\mathcal{H}_{\Omega}$ is automatically non-empty.}

\begin{theorem}
Let $\bbF_2$ be equipped with its K\"ahler--Einstein structure. There are infinitely many classes $\Omega \in \Homology^{1,1}(\bbF_2, \bbZ)$ admitting solutions $\kappa$ to the dHYM equation with $\Theta_\omega (\kappa)<\pi$ (thus not hypercritical), but that satisfy $\Im (Z_V(\Omega))>0$ for all irreducible subvarieties $V \subsetneq X$.
\end{theorem}

The classes in this result are illustrated in Figure \ref{fig:nothypercritical}.

In \cite{Chu-Lee-hypercritical-dHYM-revisited}, Chu and Lee already provided a counterexample to Conjecture \ref{conj8.7CY} in the complex $2$-dimensional setting. However, given that the conjecture had been formulated implicitly with the complex $3$-dimensional situation in mind, we find it useful to present our counterexample as well.

In \cite{Jacob-Sheu-dHYM-blowup-Pn}, the authors also write down explicit solutions to Equation \eqref{eqn:dHYM} outside the supercritical regime on the blow up of $\cp^n$ at one point. Their construction yields counter-examples to the most general version of Conjecture \ref{cjy_conjecture}. One of Jacob--Sheu's goals in \cite{Jacob-Sheu-dHYM-blowup-Pn}, is to find general algebraic stability conditions governing existence of solutions which also apply to non-supercritical $(1,1)$ classes. On the one point blow up $\blowup_p( \cp^n)$ of $\cp^n$, Jacob--Sheu establish the following.
\begin{thm}[Jacob--Sheu]
Let $\Omega$ be a class in $\Homology^{1,1}(\blowup_p( \cp^n))$ and assume $Z_X(\Omega)\neq 0$.Then there is a solution to the dHYM equation in $\Omega$ if and only if for each $k\in\{1,\ldots,n-1\}$ there is $\epsilon_k \in\{ \pm 1\}$ such that all analytic subvarieties $V^k\subset X$ of dimension $k$ satisfy 
	\begin{equation}
		\sign \Im \left( \frac{Z_{V^k}(\Omega)}{Z_{X}(\Omega)} \right) = \epsilon_k .
	\end{equation}
\end{thm}
The authors reiterate for absolute clarity (and we repeat here) that for different $k$, $\epsilon_k$ could be positive or negative but for a fixed $k$, all subvarieties of that dimension yield the same sign for $\Im \left( \frac{Z_{V^k}(\Omega)}{Z_{X}(\Omega)} \right)$.

An example of Chen in \cite{Chen-J-eqn-supercritical-dHYM} also leads to the same result and one may then wonder if the situation generalizes. One of our contributions in this article is to show that it is not the case in general. In fact, as demonstrated by Theorem \ref{thm:Intro Different Signs}, there are solutions to the dHYM equation on classes $\Omega$ and irreducible curves $C_1,C_2$ for which $ \Im \left( \frac{Z_{C_1}(\Omega)}{Z_{X}(\Omega)} \right)$ and $\Im \left( \frac{Z_{C_2}(\Omega)}{Z_{X}(\Omega)} \right)$ have different signs.
\begin{theorem}[see Corollary \ref{cor:sign quotient}]\label{thm:Intro Different Signs}
	For any K\"ahler class $K$ on $\bbF_2$ and $\Omega \in \Homology^{1,1}(\bbF_2 , \Reals)$, there is a K\"ahler form $\omega \in K$ and a $(1,1)$-form $\kappa \in \Omega$ solving the dHYM {equation}. Furthermore, there are holomorphic curves $C_1,C_2 \subset \bbF_2$ and a nonempty open set $U_K \subset \Homology^{1,1}(\bbF_2 , \Reals)$ such that
	\begin{equation}
		\sign \Im \left( \frac{Z_{C_1}(\Omega)}{Z_{\bbF_2}(\Omega)} \right) = - \sign \Im \left( \frac{Z_{C_2}(\Omega)}{Z_{\bbF_2}(\Omega)} \right) ,
	\end{equation}
	for all $\Omega \in U_K$.
\end{theorem}

While we are lacking general results concerning existence of smooth solutions outside the supercritical regime, recent progress has been made (see \cite{Murakami-weak-limits-Jflow-dHYM-on-Kahler-surfaces,Datar-Mete-Song-minimal-slopes-bubbling-cplx-Hessian-eqns,Mete-sing-formation-dHYM-cotangent-flow-blowup-CP3}) concerning the study of (potentially singular) solutions of dHYM equations in the unstable case.

\subsection{The higher rank equation}

Our main focus in this article is the higher rank deformed Hermitian--Yang--Mills equation. Let $E$ be a rank $r$ holomorphic vector bundle over $(X,\omega)$ endowed with a connection $A$ compatible with $\omega$.  We say that $A$ solves the \emph{deformed Hermitian--Yang--Mills} (dHYM) equation on $E$ with respect to $\omega$ if there is a real constant $\theta$ such that
\begin{equation}\label{eqn:rdHYM}
	\Im \left( e^{-\ii\theta}(\omega \otimes \id_E + F_A)^n \right)=0.
\end{equation}
The constant $\theta \in \Reals/\pi \bbZ$ is again determined topologically by taking the trace of both sides and integrating over $X$. 

There are examples of reducible solutions to the above equations in the work of Correa (see \cite{Correa-dHYM-higher-rank}). Correa's shows  that a connection can be both Hermitian--Yang--Mills and dHYM, or can be one of them without being the other.
Irreducible solutions were obtained in \cite{Dervan-Zcritical-Bridgeland} in the large radius regime, exploiting the fact that, in this regime, the dHYM equation ``converges'' to the HYM one. Dervan--McCarthy--Sektnan solutions occur on semistable holomorphic vector bundles. In fact, the authors show that in the large radius regime, existence of dHYM connections is equivalent to $Z$-stability, a condition which the authors define and which they prove (see Lemma (2.11) in \cite{Dervan-Zcritical-Bridgeland}) implies slope semistability. Related stability notions were introduced and studied in \cite{DelloqueNapameScarpaTipler,Keller-Scarpa-Zcritical-vectorbundles-KahlerSurfaces}. Non-trivial examples of asymptotically $Z$-stable bundles were recently given by Lara and S\'a Earp in \cite{LaraSaEarp-asymptotically-Zstable-bundles-over-proj-surfaces}.

In the small radius limit, the equation converges to the $J$-equation, which can also be studied on higher rank vector bundles. Takahashi proved in \cite{Takahashi-Jeqn-holomorphic-vector-bundles} that given an appropriate solution  to the $J$-equation, there exists a solution to the dHYM equation for sufficiently small volume.

We  make use of symmetry techniques to study both the small and large radius regime, and to find explicit dHYM connections in both regimes.

\begin{thm}[Irreducible dHYM connections in the small radius regime; see Theorem \ref{thm:Small_Radius}]
	There is a rank two complex vector bundle on the manifold $\bbF_2$ of full flags in $\bbC^3$ with the following property. For any K\"ahler form  $\omega$ on $\bbF_2$, there is $\tau>0$ such that for all $t<\tau$, there is an irreducible deformed Hermitian--Yang--Mills connection with respect to the K\"ahler form $t\omega$ on $V$. 
\end{thm}

Recall that the slope of a complex vector bundle $V$ over an algebraic manifold $X$ is defined by
\begin{equation*}
\mu(V)=\frac{\deg(V)}{\rank V} = \frac{1}{\rank(V)} \int_{X} c_1(V) \wedge \omega^2.	
\end{equation*}
The slope depends on $[\omega]$.

\begin{thm}[Irreducible dHYM connections in the large radius regime; see Theorem \ref{thm:Large_Radius}] For infinitely many $c \in \Homology^2(\bbF_2,\Reals)$, there is a rank two complex vector bundle $V=L_1 \oplus L_2$ with $c_1(V)=c$ such that for any K\"ahler form $\omega$ whose K\"ahler class lies in the open cone 
\begin{equation*}
	\mathcal{K}=\{c\in \Homology^{1,1}(\bbF_2,\Reals), \, \text{K\"ahler class} :  \mu(L_{2}) > \mu(V) > \mu(L_{1})\},
\end{equation*}
there is $T>0$ such that for all $t>T$, there is an irreducible deformed Hermitian--Yang--Mills connection with respect to the K\"ahler form $t\omega$ on $V$. 
\end{thm}

\begin{remark}
	The vector bundle $V$ splits as $L_1 \oplus L_2$ as a complex vector bundle. It does not split as a holomorphic vector bundle. Indeed, if  $V\simeq L_1 \oplus L_2$ as holomorphic vector bundles, the condition $\mu(L_{2}) > \mu(V) > \mu(L_{1})$ would imply that $V$ is unstable. However, we know from \cite{Dervan-Zcritical-Bridgeland} that this is incompatible with the existence of dHYM connections in the large radius regime.
\end{remark}

{\subsection{Organisation} In Section \ref{f2} we review the geometry of the manifold $\bbF_2$. We describe the manifold of full flags as a homogeneous manifold, study its cohomology, subvarietes, and its line bundles. At the end of the section we determine forms representing Poincar\'e duals of the homology classes determined by some relevant subvarieties.  Section \ref{1dHYM} is devoted to studying the rank 1 deformed Hermitian--Yang--Mills equation. In our setting, we show this equation always admits solutions {for any K\"ahler class and any K\"ahler form}. We then focus on Collins--Yau and Jacob--Sheu conjectures concerning existence of solutions both in and outside the supercritical/hypercritical regime. To that end we calculate central charges for relevant subvarieties. Finally, in Section \ref{2dHYM} we study the deformed Hermitian--Yang mills equation on rank $2$ bundles over $\bbF_2$ and obtain some new irreducible solutions. }

\subsection{Acknowledgments}
The first author acknowledges the support of the Natural Sciences and Engineering Research Council of Canada (NSERC), RGPIN-2019-04375. The second and third authors were partially funded by Funda\c c\~ao para a Ci\^encia e Tecnologogia (FCT) and the PRR through projects UID/04459/2025 and UID/PRR/04459/2025.

We thank Jason Lotay for bringing the references \cite{Fowdar-examples-deformed-Spin7-instantons,Fowdar-explicit-abelian-instantons-S1-inv-Kahler-6fld} to our attention.

\section{The geometry of $\bbF_{2}$}\label{f2}
The manifold $\bbF_2$ consists of the set of pairs of planes and lines in $\bbC^3$ such that the former contains the latter. Namely
\begin{equation*}
\bbF_2=\{(\mathcal{P},l) : \mathcal{P}\in \text{Gr}(2,3), l\in \cp^2,l\subset \mathcal{P}\}. 	
\end{equation*}
There are two natural projections $\proj_1\colon \bbF_2\rightarrow  \text{Gr}(2,3)\simeq \cp^2$ and $\proj_2\colon \bbF_2\rightarrow   \cp^2$.
\subsection{$\bbF_2$ as a Homogenous space}\label{ss:Hermitian structures}
As a way to understand the algebraic structure in $\bbF_2$ let us give an alternative description for the space. Consider the subvariety of $\bbC^6$ 
\begin{equation}
{\{(u,v) : u, v\in \bbC^3\setminus\{0\}\text{ and }u\cdot v=0\}}.	
\end{equation}
 There is a natural holomorphic action of ${\bbC^*\times \bbC^*}$ on this set and we define the complex structure on $\bbF_2$ so that the following map is a biholomorphism 
\begin{align*}
\frac{\{(u,v):u, v\in \bbC^3\setminus\{0\}:u\cdot v=0\}}{\bbC^*\times \bbC^*}&\rightarrow \bbF_2\\
(u,v) \,\text{mod}\, (\bbC^*)^2&\mapsto (u^\perp,\Span(v)).
\end{align*}
This structure is such that $\proj_1$ and $\proj_2$ are holomorphic. 

From this description of $\bbF_2$ we also see that  $SU(3)$ acts transitively on $\bbF_2$ with isotropy $\bbT^2$: 
\begin{equation*}
M(u,v) \,\text{mod}\, (\bbC^*)^2=(Mu,Mv) \,\text{mod}\, (\bbC^*)^2,\quad \forall M\in SU(3).	
\end{equation*}
We can identify $\bbF_2$ with $\SU(3)/ \bbT^2$, where $\bbT^2$ is the maximal torus of $\SU(3)$.  
We can be more explicit. When we write $M=\left(u\,v\,w\right)$, we mean that $u,v,w$ denote the $3$ columns of the matrix $M$.
We have the following bijection  
\begin{align*}
{\bf \Phi}\colon\SU(3)/ \bbT^2 &\rightarrow \bbF_2\\
M=\left(u\,v\,w\right) \,\text{mod}\, \bbT^2 &\mapsto (u^\perp,\Span(v)).
\end{align*}
This map shows that the homogenous manifold $\SU(3)/ \bbT^2$ admits a complex structure. 

From the homogeneous point of view, $\bbF_2$ has been studied by Bryant (see \cite{Bryant2006}).  Any $\bbT^2$-invariant tensor on $\bbF_2$ can be described via its pull back to $\SU(3)$ by the projection $\proj\colon \SU(3) \mapsto \SU(3)/\bbT^2$. 
The Maurer--Cartan $1$-form $g^{-1} dg$ on $\SU(3)$  takes values in the Lie algebra $\su(3)$ of $SU(3)$. To describe this form, we use the standard notation that $L_M$ represents left multiplication on $SU(3)$ by the matrix $M\in SU(3)$. We have for $M\in SU(3)$ and $v_M\in T_MSU(3)$ that
\begin{equation*}
(g^{-1} dg)_M(v_M)=(L_{M^{-1}})_*v_M.
\end{equation*}
 This form is $SU(3)$-left-invariant in the sense that
\begin{equation*}
( L_M)^*(g^{-1} dg)=g^{-1} dg, \; \forall M\in SU(3).	
\end{equation*}

The Lie algebra of $SU(3)$ is given by
\begin{equation*}
\su(3)=\{M\in GL(3,\bbC) :  M+\overline{M}^T=0, \text{Tr}(M)=0\}.	
\end{equation*}

Following \cite[Sec.~4.2]{Bryant2006}, we implicitly define the left-invariant 1-forms  $\lbrace \eta_i , \theta_i , \beta_i \rbrace_{i=1}^3 $ by the equation
\begin{equation}\label{eq:Maurer_Cartan_Form}
	g^{-1} dg =  \left[ \begin {array}{ccc} \ii\beta_{{1}}&\theta_{{3}}+\ii\eta_{{3}}&-
	\theta_{{2}}+\ii\eta_{{2}}\\ \noalign{\medskip}-\theta_{{3}}+\ii\eta_{{3}}
	&\ii\beta_{{2}}&\theta_{{1}}+\ii\eta_{{1}}\\ \noalign{\medskip}\theta_{{2}
	}+\ii\eta_{{2}}&-\theta_{{1}}+\ii\eta_{{1}}&\ii\beta_{{3}}\end {array}
	\right] .
\end{equation}
Note that $\beta_1+\beta_2+\beta_3=0$. The matrix-valued one-form satisfies the so-called Maurer--Cartan equation $d(g^{-1} dg)=-g^{-1} dg\wedge g^{-1} dg$. Explicitly,
\begin{equation}\label{eq:Maurer_Cartan_eqs}
\begin{aligned} 
d\beta_1&=2\left(\theta_2\wedge \eta_2-\theta_3\wedge \eta_3\right),\\
d\beta_2&=2\left(\theta_3\wedge \eta_3-\theta_1\wedge \eta_1\right),\\
d\beta_3&=2\left(\theta_1\wedge \eta_1-\theta_2\wedge \eta_2\right),\\
d\theta_1&=(\beta_2-\beta_3)\wedge \eta_1+\theta_2\wedge \theta_3-\eta_2\wedge \eta_3,\\
d\theta_2&=(\beta_3-\beta_1)\wedge \eta_2-\theta_1\wedge \theta_3+\eta_1\wedge \eta_3,\\
d\theta_3&=(\beta_1-\beta_2)\wedge \eta_3+\theta_1\wedge \theta_2-\eta_1\wedge \eta_2,\\
d\eta_1&=(\beta_3-\beta_2)\wedge \theta_1-\theta_2\wedge \eta_3-\eta_2\wedge \theta_3,\\
d\eta_2&=(\beta_1-\beta_3)\wedge \theta_2+\theta_1\wedge \eta_3+\eta_1\wedge \theta_3,\\
d\eta_3&=(\beta_2-\beta_1)\wedge \theta_3-\theta_1\wedge \eta_2-\eta_1\wedge \theta_2.
\end{aligned}
\end{equation}

For $i=1,2,3$, let
\begin{equation*}
a_i:=\theta_i+\ii \eta_i.
\end{equation*}
We summarize the last 6 equations of Equation \eqref{eq:Maurer_Cartan_eqs} by
\begin{equation}\label{eq:Maurer_Cartan_summary}
d a_i=-\ii (\beta_j-\beta_k)\wedge a_i+\overline{a_j\wedge a_k},	
\end{equation}
for cyclic permutations $(i,j,k)$ of $(1,2,3)$.

Next, we fix a set of simple roots $S=\lbrace{ r_i \rbrace}_{i=1}^3 \subset (\ttt^2)^*$ of $\SU(3)$ determined by
\begin{align*}
	r_1&=\ii\beta_2-\ii\beta_3=-\ii\beta_1 -2\ii\beta_3, \\ 
	r_2&=\ii\beta_3-\ii\beta_1, \\ 
	r_3&=\ii\beta_1 -\ii\beta_2=2\ii \beta_1+\ii\beta_3,
\end{align*}
and notice that these satisfy $r_1+r_2+r_3=0$. Then, the real component of the root spaces $\mm_i= (\mathfrak{sl}_{r_i}(3, \bbC) \oplus \mathfrak{sl}_{-r_i}(3, \bbC) ) \cap \su(3) $ are respectively given by
\begin{equation*}
	\mm_1^* = \langle \eta_1 , \theta_1 \rangle , \ \mm_2^* = \langle \eta_2 , \theta_2 \rangle , \ \mm_3^* = \langle \eta_3 , \theta_3 \rangle , 
\end{equation*}
and we shall consider $\mm$ to be the complement to the isotropy $\ttt^2 \subset \su(3)$ such that \begin{equation*}
	\mm^* =\mm^*_1 \oplus \mm^*_2 \oplus \mm^*_3.
\end{equation*}

It is useful to label elements of $\su(3)$, and more generally elements of $\mathfrak{u}(n)$. Let
\begin{equation}\label{eq:DEFmatrices}
	\begin{aligned}
		D_i&:=\diag(0,\ldots,0,\ii,0,\ldots,0),\quad\text{ (with $\ii$ in $i^{\text{th}}$ position)},\\
		E_{ij}&:=\begin{pmatrix}\text{matrix whose only non-zero entries are  }\\
						 \text{$1$ in position $(i,j)$ and $-1$ in position $(j,i)$}\end{pmatrix},\\
		F_{ij}&:=(\text{matrix whose only non-zero entries are $\ii$ in positions $(i,j)$ and $(j,i)$}).
	\end{aligned}
\end{equation}
Context should dictate the format of the matrices. We can say, for instance, that 
\[(D_1,D_2,D_3,E_{23},F_{23},E_{31},F_{31},E_{12},F_{12})\]
is a ordered basis of $\mathfrak{u}(3)$ dual to the ordered cobasis
\[(\beta_1,\beta_2,\beta_3,\eta_1,\theta_1,\eta_2,\theta_2,\eta_3,\theta_3).\] 
Let 
\begin{equation}
	\label{eq:basisofsu3}
	T_1:=\ii(D_2-D_3), \quad T_2:=\ii(D_3-D_1), \quad T_3:=\ii(D_1-D_2).
\end{equation} Though not-linearily independent, $\{T_1,T_2,T_3\}$ spans $\ttt^2$. 

The Maurer--Cartan equations allow us to verify that the forms $\theta_i,\eta_i$ do not descend to the quotient $\SU(3)/\bbT^2$. Indeed,
\begin{align}
	\mL_{T_1}\theta_1&=\iota_{T_1}d\theta_1= \iota_{T_1}\bigl((\beta_2-\beta_3)\wedge \eta_1\bigr)=2\eta_1\label{eqn:LieT1theta1},\\
	\mL_{T_1}\eta_1&=\iota_{T_1}d\eta_1= \iota_{T_1}\bigl((\beta_3-\beta_2)\wedge \theta_1\bigr)=-2\theta_1\label{eqn:LieT1eta1}.
\end{align}
Similarly, $\mL_{T_i}\theta_i=2\eta_i$ and $\mL_{T_i}\eta_i=-2\theta_i$.

We would like to consider $\SU(3)$-invariant almost complex structures $J$. 

\begin{prop}\label{lemma:invariant_complx_str}
Up to the action of the Weyl group there are only two $SU(3)$-invariant almost complex structures on $\bbF_2$, one integrable and one non-integrable.
\end{prop} 
This result is well known; see for instance \cite[Prop.~13.4 and Prop.~13.8]{Borel-Hirzebruch-characteristic-classes-homogeneous-spaces-I}). 
\begin{proof} 
Evaluating any such $J$ at the identity coset and extending it by left invariance one obtains an $(\Ad,\bbT^2)$-invariant map $J\colon \mm \to \mm$ with $J^2=-\id_{\mm}$. From Schur's lemma, any such map must preserve the root spaces $\mm_i$. We describe any such $J$ by fixing a trivialisation for the pullback bundle $p^*(\Lambda^{1,0}_{\bbC})$ over $\SU(3)$ namely
\begin{equation*}
\Span\{\alpha_1,\alpha_2,\alpha_3\}=p^\ast (\Lambda^{1,0}_{J}),	
\end{equation*}
with 
\begin{equation}\label{eq:Complex_Structures_Flag}
	\begin{aligned}
		\alpha_1 & :=  A_1(\eta_1 + \ii \epsilon_1 \theta_1), \\  
		\alpha_2 & :=  A_2(\eta_2 + \ii \epsilon_2 \theta_2), \\  
		\alpha_3 & :=  A_3(\eta_3 + \ii \epsilon_3 \theta_3),
	\end{aligned}
\end{equation}
for any $A_1,A_2,A_3 \in \Reals^+$ and $\epsilon_1, \epsilon_2, \epsilon_3 \in\Reals$.

Note that
\begin{align*}
J\eta_i&=-\epsilon_i \theta_i,\\
J\theta_i&=\frac{\eta_i}{\epsilon_i}.
\end{align*}
The invariance of $J$ implies that $\mL_{T_i}J=0$ for the elements $T_1,T_2,T_3$ of $\su(3)$ defined by Equation \eqref{eq:basisofsu3}. Using Equations \eqref{eqn:LieT1theta1} and \eqref{eqn:LieT1eta1}, we find
\begin{align*}
	0&=J(\mL_{T_1}\theta_1)-\mL_{T_1}(J\theta_1)\\
	 &= 2J\eta_1-\mL_{T_1}\bigl(-\frac{\eta_1}{\epsilon_1}\bigr)\\
	 &=-2\epsilon_1\theta_+\frac{2}{\epsilon_1}\theta_1\\
	 &=\frac{2(1-\epsilon_1^2)}{\epsilon_1}\theta_1,
\end{align*}
and thus $\epsilon_1=\pm1$. Cyclicly permuting $(1,2,3)$ in the above proof yields similarly that $\epsilon_2=\pm1$ and $\epsilon_3=\pm1$.

Recall that the Weyl group of $SU(3)$ is generated by the reflections $p_1,p_2,p_3$ defined via 
\begin{align*}
p_1(r_1,r_2,r_3)&=(-r_1,-r_3,-r_2),\\
p_2(r_1,r_2,r_3)&=(-r_3,-r_2,-r_1),\\
p_3(r_1,r_2,r_3)&=(-r_2,-r_1,-r_3).
\end{align*}
In particular, $\sigma = p_2 \circ p_1$ is an element of order $3$ which cyclically permutes the roots $r_1,r_2,r_3$.
The Weyl group is acting the quotient  $\bbF_2$, and 
 induces an action on the space of almost complex structures determined by $(\epsilon_1,\epsilon_2,\epsilon_3)$.  Up to the action of the Weyl group there are only two invariant almost complex structures, one where all signs $\epsilon_i$ are equal, and one where they are not all equal.

Given an almost complex structure $J$ determined by (\ref{eq:Complex_Structures_Flag}), we can determine whether it is integrable by computing the projection $\pi^{0,2}d\alpha_i$ of $d\alpha_i$ on $p^*\Lambda^{0,2}_\bbC$. For any cyclic permutation $(i,j,k)$ of $(1,2,3)$, we have
\begin{align*}
	\pi^{0,2}d\alpha_i &= -\ii\frac{(\epsilon_1\epsilon_2\epsilon_3+\epsilon_1+\epsilon_2+\epsilon_3)}{4\epsilon_j\epsilon_k}\frac{A_i}{A_jA_k}\overline{\alpha}_j\wedge\overline{\alpha}_k.
\end{align*}
Those projections must be zero for the almost complex structure to be integrable. Thus $J$ is integrable if and only if
\begin{equation}\label{eq:Integrable_Flag}
	\epsilon_1\epsilon_2\epsilon_3+\epsilon_1+\epsilon_2+\epsilon_3 =0.
\end{equation}
 Inserting $(\epsilon_1,\epsilon_2,\epsilon_3)=(1,-1,1)$ into Equation \eqref{eq:Integrable_Flag}, we find an integrable complex structure, and inserting $(\epsilon_1,\epsilon_2,\epsilon_3)=(1,1,1)$ we find an non-integrable one. 
\end{proof}

\begin{definition} Let  $J^i$ be the integrable complex structure obtained by taking $(\epsilon_1,\epsilon_2,\epsilon_3)=(1,-1,1)$ and $J^{ni}$ be the non-integrable complex structure obtained by taking $(\epsilon_1,\epsilon_2,\epsilon_3)=(1,1,1)$.
\end{definition}

 Now we consider the almost Hermitian structure determined by setting
\begin{equation}\label{eq:SU(3)structure_Flag1}
	\omega :=  \frac{\ii}{2} (\alpha_1 \wedge \overline{\alpha}_1 + \alpha_2 \wedge \overline{\alpha}_2 + \alpha_3 \wedge \overline{\alpha}_3 ).
\end{equation}
Note that $\omega$ descends to a two-form on $\bbF_2$. In fact, every $\alpha_i\wedge\overline\alpha_i$ descend to $\bbF_2$.

\begin{lem}\label{lemma:closed_form}
The form $\omega$ on $\bbF_2$ defined above is symplectic if and only if 
\begin{equation}\label{eq:Symplectic_Flag}
	A_1^2 \epsilon_1 + A_2^2 \epsilon_2 + A_3^2 \epsilon_3 =0.
\end{equation}
\end{lem}
\begin{proof}
For any 1-form $\alpha$, we have $d(\ii\alpha\wedge\overline\alpha)=2\Im(d\alpha\wedge \overline\alpha)$. 
Note that 
\begin{equation*}
	\alpha_i =\begin{cases}
			A_i a_i,&\text{ if }\epsilon_i=1,\\
			A_i \overline{a}_i,&\text{ if }\epsilon_i=-1,
	   		\end{cases}
\end{equation*}	
and thus	$d(\ii\alpha_i\wedge\overline\alpha_i) = 2\epsilon_iA_i^2\Im(da_i\wedge\overline{a_i})$. We thus have
\begin{equation*}
	d\omega = \Im\bigl(\epsilon_1A_1^2 da_1\wedge\overline{a_1}+\epsilon_2A_2^2 da_2\wedge\overline{a_2}+\epsilon_3A_3^2 da_3\wedge\overline{a_3}\bigr).
\end{equation*}

Using the Maurer--Cartan Equation \eqref{eq:Maurer_Cartan_summary}, we see that for cyclic permutations $(i,j,k)$ of $(1,2,3)$, we have
\[da_i\wedge \overline{a_i}=-\ii (\beta_j-\beta_k)a_i\wedge \overline{a_i}+\overline{a_1\wedge a_2\wedge a_3}.\] 
As $\ii (\beta_j-\beta_k)a_i\wedge \overline{a_i}$ is always real, we have
\begin{equation}
	d\omega = (\epsilon_1 A_1^2+\epsilon_2A_2^2+\epsilon_3A_3^2)\Im(\overline{a_1\wedge a_2\wedge a_3}).
\end{equation} 
Given that $\Im(\overline{a_1\wedge a_2\wedge a_3})=\eta_1\wedge\eta_2\wedge\eta_3-\eta_1\wedge\theta_2\wedge\theta_3-\theta_1\wedge\eta_2\wedge\theta_3-\theta_1\wedge\theta_2\wedge\eta_3$ is non-zero, we see that $\omega$ is closed if and only if Equation \eqref{eq:Symplectic_Flag} is satisfied.
\end{proof}

Any choice of $(\epsilon_1,\epsilon_2,\epsilon_3)$ and $(A_1,A_2,A_3)$ satisfying both Equations \eqref{eq:Integrable_Flag} and \eqref{eq:Symplectic_Flag} then yields a homogeneous K\"ahler structure on $\bbF_2$.  Thus $J^i$ is compatible with a real $2$-parameter family of K\"ahler structures determined by solving the above Equation \eqref{eq:Symplectic_Flag}.  On the other hand, $J^{ni}$ is compatible with a nearly K\"ahler structure.

\begin{exam}[K\"ahler--Einstein structure]\label{ex:Kahler-Einstein}
	The K\"ahler structure determined by $J^i$ and $\omega$ with $A_1^2=A^2=A_3^2$ and $A_2^2=2A^2$ is the standard homogeneous K\"ahler--Einstein structure on $\bbF_2$.
\end{exam}

\subsection{Cohomology of $\bbF_2$}
The cohomology of $\bbF_2$ is well known with $h^0=h^6=1$, $h^2=h^4=2$, and $h^1=h^3=h^5=0$; see for instance \cite[Proposition 21.17]{BottTu}.

It is useful at this time to introduce the notation $\sumcyclic$:  
\begin{equation}
	\sumcyclic f(i,j,k):=f(1,2,3)+f(2,3,1)+f(3,1,2).
\end{equation}

\begin{lem}\label{lem:cohomology}
	All cohomology classes of $\bbF_2$ can be represented by $\SU(3)$-invariant representatives:
	\begin{align}
		\Homology^2(\bbF_2,\Reals) &= \biggl\{[x_1\ii \alpha_1\wedge \overline{\alpha}_1+x_2\ii \alpha_2\wedge \overline{\alpha}_2+x_3\ii \alpha_3\wedge \overline{\alpha}_3] : \sum_{i=1}^3 x_i\epsilon_i A_i^2=0\biggr\},\label{eqn:H2}\\
		\Homology^4(\bbF_2,\Reals)& = \Biggl\{ \biggl[\sumcyclic x_i\alpha_j\wedge \overline{\alpha}_j\wedge \alpha_k\wedge \overline{\alpha}_k\biggr]\Biggr\}.
	\end{align}
In particular, $\Homology^{1,1}(\bbF_2,\Reals)=\Homology^2(\bbF_2,\Reals)$ in the presence of a K\"ahler structure.
\end{lem}

\begin{proof}Let $\gamma_i:=a_i\wedge \overline{a_i}$.  A linear combination $x_1\gamma_1+x_2\gamma_2+x_3\gamma_3$ is closed if $x_1+x_2+x_3=0$. 
We introduce closed forms $\nu_0:=\gamma_1+\gamma_2-2\gamma_3$ and $\nu_i:=\gamma_j-\gamma_k$. We have $\nu_0\wedge\nu_0\wedge\nu_0=-12\gamma_1\wedge\gamma_2\wedge\gamma_3$, and $\nu_0\wedge\nu_3\wedge_3 = 4\gamma_1\wedge\gamma_2\wedge\gamma_3$. Thus both cohomology classes $[\nu_0]$ and $[\nu_3]$ are non-zero. Since $\nu_i^3=0$, the cohomology classes $[\nu_0]$ and $[\nu_3]$ must be linearly independent. Since $h^2=2$, it must be that $\Homology^2(\bbF_2,\Reals)$ is spanned by those classes. Any linear combination is of the form $x_1\gamma_1+x_2\gamma_2+x_3\gamma_3$ with $x_1+x_2+x_3=0$. Equation \eqref{eqn:H2} follows from the fact that $\alpha_i \wedge \overline{\alpha}_i = \epsilon_iA_i^2\gamma_i$. 

Let $\sigma_i:=\gamma_j\wedge \gamma_k$ for any cyclic permutation $(i,j,k)$ of $(1,2,3)$. 
Given that the $\sigma_i$ are automatically closed, they represent classes in $\Homology^4(\bbF_2,\Reals)$. Since $\gamma_i\wedge \sigma_j = \delta_{ij}\gamma_1\wedge\gamma_2\wedge\gamma_3$, the classes $[\sigma_i]$ are all non-zero. Suppose that $0=x_1[\sigma_1]+x_2[\sigma_2]$. Then given that $[(x_1\sigma_1+x_2\sigma_2)\wedge \nu_1]=x_2[\gamma_1\wedge\gamma_2\wedge\gamma_3]$, we must have $x_2=0$, and thus $x_1=0$. So $[\sigma_1]$ and $[\sigma_2]$ are linearly independent and thus must span $\Homology^4(\bbF_2,\Reals)$. 
\end{proof}

\subsection{Subvarieties in $\bbF_2$}
{In this section we assume that $\bbF_2$ is endowed with its usual complex structure. From Proposition \ref{lemma:invariant_complx_str} (see also \cite{Borel-Hirzebruch-characteristic-classes-homogeneous-spaces-I}) we know that any two $SU(3)$-invariant complex structures on $\bbF_2$ are biholomorphic and a complex structure is determined by a Weyl chamber so that the usual complex structure we are considering is $J^i$ from the previous section up to the action of the Weyl group.}
\subsubsection{Curves}
Consider the following subgroups $H_i \subset \SU(3)$.
\begin{align*}
	H_1 &= \left\{ \begin{pmatrix}
		a & 0 & 0 \\
		0 & b & c \\
		0 & d & e
	\end{pmatrix} \in \SU(3) \ : \  \begin{pmatrix}
		b & c \\
		d & e
	\end{pmatrix} \in \U(2)  \right\} , \\
	H_2 &= \left\{ \begin{pmatrix}
		a & 0 & c \\
		0 & b & 0 \\
		d & 0 & e
	\end{pmatrix} \in \SU(3) \ : \  \begin{pmatrix}
		a & c \\
		d & e
	\end{pmatrix} \in \U(2)  \right\} , \\
	H_3 &= \left\{ \begin{pmatrix}
		a & c & 0 \\
		d & b & 0 \\
		0 & 0 & e
	\end{pmatrix} \in \SU(3) \ : \  \begin{pmatrix}
		a & c \\
		d & b
	\end{pmatrix} \in \U(2)  \right\} .
\end{align*}
For $i=1,2,3$, consider the curves $C_i$ obtained as the orbits of the subgroups $H_i \subset \SU(3)$ at the coset of the identity. 
{\begin{lem}\label{lemma:C_i}
The curves $C_i$, $i=1,2,3$, are complex projective lines in $\bbF_2$ for its usual integrable complex structure
\end{lem}
\begin{proof}
The $C_i$ correspond to Schubert varieties in $\bbF_2$ and are projective lines. Without loss of generality we focus on the $i=3$ case to verify these claims. We have
\[{\bf \Phi}(H_3\cdot I)=\{(u^\perp,\Span(v)) : u,v\in (0,0,1)^\perp\}.\]
For any $u\in (0,0,1)^\perp$, we write $u=(u_1,u_2,0)\in \C^3$.
The set $\{(u^\perp,\Span(v)) : u,v\in (0,0,1)^\perp\}$ is algebraic in $\bbF_2$ and biholomorphic to $\cp^1$ via
\[\begin{aligned}
 \{(u^\perp,\Span(v)) : u,v\in (0,0,1)^\perp\}&\rightarrow &\cp^1\\
 (u^\perp,\Span(v))&\mapsto& [u_1:u_2].
 \end{aligned}\]
The proof is now complete.
\end{proof}}

\begin{lem}[Poincar\'e duals of curves]\label{lem:PD curves}	Let $(i,j,k)$ be a cyclic permutation of $(1,2,3)$. Let $\mu_i:=\frac{\Vol(C_i)}{\Vol(\bbF_2)}$.  The degree $4$ cohomology class $PD[C_i]$ Poincar\'e dual to the curve $C_i$ is represented by the $4$-form
	\[\alpha_{C_i}=\mu_i (\frac{\ii}{2} \alpha_j \wedge \overline{\alpha_j}) \wedge (\frac{\ii}{2} \alpha_k \wedge \overline{\alpha_k}) . \]
	 Moreover, there is $c>0$, independent of $i$, such that $\mu_i=\frac{c}{A_j^2A_k^2}$, thus
	\[\alpha_{C_i}= \frac{c}{A_j^2 A_k^2} \ (\frac{\ii}{2} \alpha_j \wedge \overline{\alpha_j} ) \wedge (\frac{\ii}{2} \alpha_k \wedge \overline{\alpha_k}) .\]
\end{lem}
\begin{proof}
	The tangent space to these $H_i$-orbits corresponds to the (left-translation of the) root spaces $\mm_i$ and therefore, we find that for any $i,j \in \lbrace 1,2,3 \rbrace$,
	\[(\tfrac{\ii}{2} \alpha_j \wedge \overline{\alpha_j})|_{\mm_i} = \delta_{ij} \alpha_j \wedge \overline{\alpha_j}  = \delta_{ij} \epsilon_i A_i^2 \theta_i \wedge \eta_i.\]
	Furthermore, because holomorphic curves are calibrated by the K\"ahler form $\omega$, we must have $\omega|_{\mm_i} = \vol_{C_i}$ from which we infer that
	\[\Vol(C_i) = \int_{C_i} \omega = \epsilon_i A_i^2 \int_{C_i}\theta_i \wedge \eta_i,\]
	and therefore
	\[\int_{C_i}\theta_i \wedge \eta_i = \frac{\Vol(C_i)}{\epsilon_i A_i^2}. \]
	Let $\gamma = \frac{\ii}{2} \sum_{i=1}^3 \left(\gamma_i \alpha_i \wedge \overline{\alpha_i}\right)$ be an invariant closed $2$-form on $\bbF_2$. Note that  $\gamma$ is closed if and only if $\sum_{i=1}^3 \epsilon_i A_i^2 \gamma_i =0$. 
	For all $i=1,2,3$
	\[\int_{C_i} \gamma =  \epsilon_i A_i^2 \gamma_i \int_{C_i} \eta_i \wedge \theta_i  = \gamma_i \Vol(C_i) .\]
 Therefore, the closed invariant $4$-form 
	\[\alpha_{C_l}= \sum_{i=1}^3 c_{i} (\frac{\ii}{2} \alpha_j \wedge \overline{\alpha_j}) \wedge (\frac{\ii}{2} \alpha_k \wedge \overline{\alpha_k})  , \]
	represents the Poincar\'e dual to $C_l$ if and only if
	\begin{align*}
		\gamma_l \Vol(C_l)  & = \int_{\bbF_2} \alpha_{C_l} \wedge \gamma \\
		& = \int_{\bbF_2} \bigl(\sum_{i=1}^3 c_{i} \gamma_i\bigr)  (\frac{\ii}{2} \alpha_1 \wedge \overline{\alpha_1}) \wedge (\frac{\ii}{2} \alpha_2 \wedge \overline{\alpha_2}) \wedge (\frac{\ii}{2} \alpha_3 \wedge \overline{\alpha_3}) \\
		& =  \Vol(\bbF_2)\sum_{i=1}^3 c_i \gamma_i  ,
	\end{align*}
	for all $(\gamma_1,\gamma_2,\gamma_3)$ such that $\sum_{i=1}^3 \epsilon_i A_i^2 \gamma_i = 0$. Summarising, $\alpha_{C_l}$ is a representative of $PD[C_l]$ if and only if 
	\[\sum_{i=1}^3 c_i \gamma_i =  \gamma_l \frac{\Vol(C_l)}{\Vol(\bbF_2)}.\]
	A solution to this equation consists is obtained by setting
	\[c_l = \delta_{il}\frac{\Vol(C_l)}{\Vol(\bbF_2)} .\]
	Furthermore, one can check that the root spaces $\mm_1$, $\mm_2$, $\mm_3$ are interchanged by the action of the Weyl group and this action preserves the forms $\theta_i \wedge \eta_i$ up to sign. Thus $I=|\int_{C_i} \theta_i \wedge \eta_i|$ is independent of $i=1,2,3$, and therefore
	\[\Vol(C_i) = A_i^2 I ,\]
	for all $i=1,2,3$. Since $\Vol(\bbF_2) \sim A_1^2 A_2^2 A_3^2$, we find that
	\[\mu_i =  \frac{c}{A_j^2 A_k^2},\]
	for a positive constant $c>0$.
\end{proof}

\subsubsection{Surfaces}\label{ss:U(1) bundles}
We shall construct surfaces in $\bbF_2$ as zero-sets of holomorphic section over holomorphic line bundles so we start with a discussion of such bundles.

Circle bundles are topologically classified by the first Chern class of the line bundle associated with the standard representation. The first Chern class is an element of $\Homology^2(\bbF_2, \bbZ)$, a cohomology group which can be computed using the Serre spectral sequence for the fibration $\bbT^2 \rightarrow \SU(3) \rightarrow \bbF_2$.
Let $\Lambda:=\ker(\exp\colon \ttt^2 \rightarrow \bbT^2)$ be the cocharacter lattice. The cohomology $\Homology^1(\bbT^2 , \bbZ)$ is isomorphic to the weight lattice $\Hom (\Lambda , \bbZ)$. Together, these two ideas yield
\begin{equation}\label{eq:Serre}
	\Homology^2(\bbF_2, \bbZ) \cong \Homology^1(\bbT^2 , \bbZ)=\Hom (\Lambda , \bbZ).
\end{equation}

An integral weight $\beta \in \Hom (\Lambda , \bbZ) \subset \Hom(\ttt^2 , \Reals)$ can be viewed as a $1$-form in $\ttt^2$ and can be extended to $\su(3)$ as $\beta \oplus 0$ using the orthogonal splitting induced by the Killing form. Finally, it can be extended to a left-invariant $1$-form in $\SU(3)$, which we still denote $\beta$. 

Let $\beta=(n_1,n_2)\in \bbZ^2$ and consider the group homomorphism 
\begin{align*}
\lambda_\beta\colon \bbT^2&\rightarrow U(1)=S^1\\
(\theta_1,\theta_2)&\mapsto n_1 \theta_1 + n_2 \theta_2.
\end{align*}
Now consider the principal $U(1)$-bundle associated to the principal $\bbT^2$-bundle $\proj\colon SU(3)\rightarrow \bbF_2$ and $\lambda_\beta$ that is
\[SU(3)\times_{\lambda_\beta} U(1) \rightarrow \bbF_2.\]
Its total space is $SU(3)\times U(1)$ modulo the equivalence relation $\sim$ given by
\[
(M,e^{\ii\theta})\sim 
\left(M
\begin{pmatrix}
e^{\ii \theta_1}&0&0\\
0&e^{\ii \theta_2}&0\\
0&0&e^{-\ii (\theta_1+\theta_2)}
\end{pmatrix},
e^{-\ii\left(n_1 \theta_1 + n_2 \theta_2\right)}e^{\ii\theta}\right).
\]
The imaginary valued $1$-form $\ii \beta$ in $\SU(3)$ can be viewed as a connection on $SU(3)\times_{\lambda_\beta} U(1)$. That it is so follows from a theorem of Wang (see \cite{Wang-invariant-connections} or Theorem \ref{thm:Wang}) but, in fact, it is quite easy to see it directly in this special case as $U(1)$ is abelian and $\ii \beta$ restricts to the image on the fibres of the Maurer--Cartan form on $U(1)$ which is $\ii d\theta$. Let $L_\beta$ denote the corresponding line bundle
\[L_{\beta}=\SU(3) \times_{ e^{i\beta} } \bbC.\]
Using Chern--Weil theory, the first Chern class of $L_\beta$ is the cohomology of its curvature form, namely
\[c_1(L_{\beta}) = \frac{\ii}{2 \pi} [d ( \ii \beta )] = - \frac{1}{2 \pi} [d \beta ] \in  \Homology^2(\bbF_2, \bbZ).\]
Under the isomorphism of Equation \ref{eq:Serre}, $c_1(L_{\beta})$ corresponds to $ \beta \in \Hom (\Lambda , \bbZ) $.

For $m=(m_1,m_3) \in \bbZ^2$, let 
\begin{equation}
	m_2:=-(m_1+m_3),
\end{equation}
and
\begin{equation}\label{eq:betam}
	\beta_m:=m_1 \beta_3 -  m_3 \beta_1.
\end{equation}
The choice of labels in this equation may appear strange, but it is made in order to make some formulas ahead more symmetric.  The curvature of the connection $\ii d\beta_m$ is obtained explicitly using the Maurer--Cartan equation:
\begin{equation}\label{eq:curvature U(1)}
	F_{m}  =-\frac{1}{2}\left(\frac{2m_1}{\epsilon_1 A_1^2}  \alpha_1 \wedge \overline{\alpha_1} + \frac{2m_2}{\epsilon_2 A_2^2}  \alpha_2 \wedge \overline{\alpha_2} + \frac{2m_3}{\epsilon_3 A_3^2}  \alpha_3 \wedge \overline{\alpha_3} \right).
\end{equation}
In particular, it can be written as the $(1,1)$-form $\ii \kappa_m$ with
\begin{align}
	\kappa_m &:= \sum_{i=1}^3 \lambda_i \frac{\ii}{2} \alpha_i \wedge \overline{\alpha_i},\\
	\lambda_i &= - \frac{2m_i}{\epsilon_i A_i^2}, \ \text{for $i=1,2,3$}.
\end{align} 
We see that $L_\beta$ is a holomorphic bundle.

As $m$ varies in $\Reals^2$ (rather than in $\bbZ^2$), $[\kappa_m]$ spans $\Homology^2(\bbF_2,\Reals)$ as per Lemma \ref{lem:cohomology}. We use the map $m=(m_1,m_3)\mapsto [\kappa_m]$ to give coordinates on $\Homology^2(\bbF_2,\Reals)$.

\begin{remark}
As it is explained in \cite{ABBS}, every line bundle over $\bbF_2$ is isomorphic to one of the bundles $\mO(n_1,n_2):=\proj_1^*\mO(n_1)\otimes \proj_2^*\mO(n_2)$ where $\mO(1)$  denotes the tautological bundle over $\cp^2$. In our case it is not hard to see that
 \[L_{\beta_m} = \mathcal{O} (m_3,-(m_1+m_3)).\]
In fact, having in mind the previously considered groups $H_i$ for $i=1,3$, and noticing that $H_1\simeq S (\U(1) \times \U(2))$ and $H_3\simeq S (\U(2) \times \U(1))$, one can consider the following group homomorphisms $\mu_{i} \colon H_i \to \U(1)$:
\begin{align*}
	\mu_1 (e^{ix}, A) & = e^{ix} \\
	\mu_3 (A,e^{ix}) & = e^{ix} .
\end{align*}
Then, composing these with the inclusions $\iota_i \colon \bbT^2 \to H_i$ and multiplying them we obtain the group homomorphism
\[\lambda_m = (\mu_1^{-m_3} \circ\iota_1 ) \otimes (  \mu_3^{m_1} \circ \iota_3).\]
Having this in consideration, we find that
\[L_{\beta_m} = \mO (m_3,-(m_1+m_3)) := \proj_1^* \mO (m_3) \otimes \proj_2^*\mO(-(m_1+m_3)) .\]
\end{remark}

\begin{lem}\label{lemma:positive_bundle}
The Chern class $c_1(L_{\beta_m})$ is semi-positive iff both $m_1$ and $m_3$ are non-negative and it is positive if it is semi-positive and at least one the $m_i$ is strictly positive. 
\end{lem}
	\begin{proof}As we have said before $c_1(L_{\beta_m})$ is represented by the form $\frac{\ii F_m}{2\pi} =-\frac{d \beta_m}{2 \pi}$ which is given by
\[ \frac{\ii }{4\pi} \left( \frac{2m_1}{\epsilon_1 A_1^2}  \alpha_1 \wedge \overline{\alpha_1} + \frac{2m_2}{\epsilon_2 A_2^2} \alpha_2 \wedge \overline{\alpha_2} + \frac{2m_3}{\epsilon_3 A_3^2} \alpha_3 \wedge \overline{\alpha_3} \right).\]
	This form is a semi-positive form if and only if $\epsilon_i m_i\geq 0$ for $i=1,2,3$. We have $(\epsilon_1,\epsilon_2,\epsilon_3)=(1,-1,1)$. Our condition amounts to having both $m_1$ and $m_3$ non-negative. The form is then be positive if in addition at least one of the $m_i$ is positive.
\end{proof}

Because the pairs $(0,d)$ and $(d,0)$ satisfy the conditions in Lemma \ref{lemma:positive_bundle}, the Chern forms of  $L_{\beta_{(0,d)}}$, $L_{\beta_{(d,0)}}$ are positive and the bundles $L_{\beta_{(0,d)}}$, $L_{\beta_{(d,0)}}$ are ample. By the Kodaira embedding theorem such bundles admit holomorphic sections whose zero set define homology classes which are Poincaré dual to the Chern classes of the bundles. Hence we find that the classes induced by the forms
\begin{align*}
	\frac{\ii}{2\pi}F_{(d,0)} = -\frac{d}{2\pi} d\beta_3 = \frac{\ii d}{4\pi} \left( \frac{2}{A_1^2}  \alpha_1 \wedge \overline{\alpha_1} + \frac{2}{ A_2^2}\alpha_2 \wedge \overline{\alpha_2} \right), \\
	\frac{\ii}{2\pi}F_{(0,d)} = \frac{d}{2\pi} d\beta_1 =  \frac{\ii d}{4\pi} \left( \frac{2}{A_2^2}  \alpha_2 \wedge \overline{\alpha_2} + \frac{2}{A_3^2}  \alpha_3 \wedge \overline{\alpha_3} \right) ,
\end{align*}
are represented by subvarieties. Indeed, we find that these classes are Poincar\'e dual to the pullbacks of degree-$d$ curves in either $\cp^2$ or $\text{Gr}(2,3)$. 
Another class whose Poincar\'e dual can be represented by an irreducible hypersurface is that of $\frac{\ii}{2\pi}F_{(1,1)}$ which is given by
\[\frac{\ii}{2\pi}F_{(1,1)} =  \frac{\ii}{4\pi} \left( \frac{2}{A_1^2} \alpha_1 \wedge \overline{\alpha_1} + \frac{4}{A_2^2}  \alpha_2 \wedge \overline{\alpha_2} + \frac{2}{A_3^2}  \alpha_3 \wedge \overline{\alpha_3} \right).\]

\begin{lem}[Poincar\'e duals of hypersurfaces]\label{lem:PD hypersurface}
	If the $2$-form
	\[\alpha_\surface = \frac{\ii}{2}\sum_{i=1}^3 \mu_i  \alpha_i \wedge \overline{\alpha_i},\]
	represents the Poincar\'e dual of a (complex) hypersurface, then $\sum_{i=1}^3 \epsilon_i A_i^2 \mu_i=0$ and both $\mu_1$, $\mu_3$ must be non-negative (and so must $\mu_2$ because $\epsilon_1=1=\epsilon_3$ and $\epsilon_2=-1$).
\end{lem}
\begin{proof}
	Using the Maurer--Cartan Equation \eqref{eq:Maurer_Cartan_eqs}, we find that
\[d\alpha_\surface=\frac{\ii}{2}(\mu_1A_1^2-\mu_2A_2^2+\mu_3A_3^2)\overline{a_1\wedge a_2\wedge a_3}.\]
Thus, as $\alpha_\surface$ is closed, we must have \[\sum_{i=1}^3 \epsilon_i A_i^2 \mu_i=0.\] 
Let $\surface$ be a hypersurface. Then, for any K\"ahler form
	\begin{align*}
		0 & < \Vol(\surface) = \int_\surface \frac{\omega^2}{2} = \int_X \alpha_\surface \wedge \frac{\omega^2}{2} =  \biggl(\sum_{i=1}^3 \mu_i \biggr)\int_X \frac{\omega^3}{3!}.
	\end{align*}
	Hence, $\sum_{i=1}^3 \mu_i>0$. Using the closedness to eliminate $\mu_2$, we find that this condition turns into
	\[\mu_1 \left(1+\frac{A_1^2}{A_2^2}\right) + \mu_3 \left(1+\frac{A_3^2}{A_2^2}\right) >0.\]
	If either $\mu_1$ or $\mu_3$ was to be negative, say $\mu_1$, then we could pick $\frac{A_1^2}{A_2^2}$ sufficiently large and $\frac{A_3^2}{A_2^2}$ sufficiently small so that the left hand side was negative, thus leading to a contradiction.
\end{proof}

\section{The rank one dHYM equations}\label{1dHYM}
To be absolutely clear, let us note that for the rest of the paper we use the integrable complex structure $J^i$ defined by setting $(\epsilon_1,\epsilon_2,\epsilon_3)=(1,-1,1)$. We nonetheless will continue to carry along various $\epsilon_i$ in many computations without specializing as doing so simplifies notation.

\subsection*{The phase of solutions}
Let us return to the dHYM equations (\ref{eqn:dHYM}) on $(\bbF_2,\omega)$. We shall consider invariant $(1,1)$-forms given by 
\begin{equation}\label{eq:kappa}
	\kappa = \frac{\ii}{2} \left( \lambda_1 \alpha_1 \wedge \overline{\alpha_1} + \lambda_2 \alpha_2 \wedge \overline{\alpha_2} + \lambda_3 \alpha_3 \wedge \overline{\alpha_3} \right),
\end{equation}
for $\lambda_i \in \Reals$ constant. In particular the sum of the $\arctan(\lambda_i)$ is also constant. Therefore any such $\kappa$ yields a solution to the dHYM equations (\ref{eqn:dHYM}).

We shall now prove Theorem \ref{thm:Surjective_INTRO} which we restate here in an equivalent form.

\begin{theorem}\label{thm:Theta_surjective}
	Let $\omega$ be a homogeneous K\"ahler metric on $\bbF_2$. Then, for any $\Omega \in \Homology^{1,1}(\bbF_2, \Reals)$, there is a solution $\kappa_\Omega \in \Omega$ of the dHYM equations.
	Furthermore, the map  
	\begin{align*}
		\Homology^{1,1}(\bbF_2, \Reals) \to \bigl(-\frac{3\pi}{2} , \frac{3\pi}{2}\bigr);\quad
		\Omega\mapsto \Theta_\omega(\kappa_\Omega)
	\end{align*}
	is surjective.
\end{theorem}
\begin{proof}
	From Equation \eqref{eq:curvature U(1)}, the curvature of the connection $\ii \beta$ on the $\U(1)$-bundle $L_\beta$ over $\bbF_2$ can be written as the $(1,1)$-form $\ii \kappa$ for $\kappa$ as in \eqref{eq:kappa} with 
	\[\lambda_i = - \frac{2m_i}{\epsilon_i A_i^2}, \ \text{for }i=1,2,3.\]
	As a function over $\bbF_2$, $\lambda_i$ is constant and so we obtain a solution for any such class. We can compute the resulting phase using $\epsilon_1=1=\epsilon_3$, $\epsilon_2=-1$, $m_1+m_2+m_3=0$, and $A_2^2=A_1^2+A_3^2$, which gives
	\begin{align*}
		\Theta_{\omega}(-\ii F_m) & = - \sum_{i=1}^3 \arctan \left( \frac{2m_i}{\epsilon_i A_i^2}\right)  \\
		& = - \arctan \left(\frac{2m_1}{A_1^2} \right)  + \arctan \left( \frac{2m_2}{A_2^2} \right)  - \arctan \left( \frac{2m_3}{A_3^2} \right)  \\
		& = - \arctan \left( \frac{2m_1}{A_1^2} \right)  - \arctan \left( \frac{2m_1+2m_3}{A_1^2+A_3^2} \right)  - \arctan \left( \frac{2m_3}{A_3^2} \right) .
	\end{align*}
	Furthermore, we can consider real $(1,1)$-classes $[\kappa_m]$ by having $m=(m_1,m_3) \in \Reals^2$. Then, the above function $m \mapsto \Theta_{\omega}(\kappa_m)$ is continuous and we have
\[\lim_{m_1,m_3 \to \pm \infty} = \mp \frac{3\pi}{2},\]
	and so all intermediate values are achieved.
\end{proof}

\subsection{Preliminaries on central charges}
Direct computation gives
\[\frac{(\omega+ \ii \kappa)^3}{\omega^3} = \prod_{i=1}^3 (1+ \ii\lambda_i)  = \biggl(\prod_{i=1}^3 \sqrt{1+ \lambda_i^2 }\biggr) \exp \biggl(\ii \sum_{i=1}^3 \arctan (\lambda_i) \biggr).\]
In particular $\Re (e^{-\ii\theta} (\omega+ \ii \kappa)^3 ) > 0$ if and only if 
\[\sum_{i=1}^3 \arctan (\lambda_i) - \theta \in \biggl(-\frac{\pi}{2} , \frac{\pi}{2}\biggr) \ \mod 2 \pi .\]
On the other hand $\Im (e^{-\ii \theta} (\omega+ \ii \kappa)^3) = 0$ if and only if
\[\sum_{i=1}^3 \arctan (\lambda_i) = \theta \mod \pi .\]
We therefore conclude that the condition that $\Im (e^{-\ii\theta} (\omega+ \ii \kappa)^3 )= 0$ and $\Re (e^{-\ii\theta} (\omega+ \ii \kappa)^3 )> 0$ means that 
\begin{equation}\label{eq:dHYM_arctan}
	\sum_{i=1}^3 \arctan (\lambda_i) = \theta  \mod 2 \pi ,
\end{equation}
and a solution $\kappa$ to this equation is called supercritical if $\theta= \sum_{i=1}^3 \arctan(\lambda_i) \in \left(\frac{\pi}{2} ,  \frac{3\pi}{2}\right)$, and hypercritical if $\theta \in \left( \pi ,  \frac{3\pi}{2}\right)$. Notice that by changing all $\lambda_i$ to $-\lambda_i$, the angle $\theta$ changes to $- \theta$. Hence, it also make sense to consider the case in which $|\theta|> \frac{\pi}{2}$ and we shall refer to these solutions as supercritical up to sign.

\begin{remark}
	We note that
	\begin{equation}\label{eq:product}
		\prod_{i=1}^3 (1+ \ii\lambda_i) = \left(1- \sumcyclic \lambda_j \lambda_k\right) + \ii \left( \sum_{i=1}^3\lambda_i - \lambda_1 \lambda_2 \lambda_3 \right).
	\end{equation}
As a consequence,
	\[\sum_{i=1}^3 \arctan (\lambda_i) = \arctan \left( \frac{\sum_{i=1}^3\lambda_i - \lambda_1 \lambda_2 \lambda_3}{ 1- \sumcyclic  \lambda_j \lambda_k } \right) \mod \pi,\]
and
\begin{equation}\label{eq:tan_three_arctan}
	\tan\biggl(\sum_{i=1}^3 \arctan (\lambda_i)\biggr) =  \frac{\sum_{i=1}^3\lambda_i - \lambda_1 \lambda_2 \lambda_3}{ 1- \sumcyclic \lambda_j \lambda_k }.
\end{equation}
This formula is particularly helpful to understand the boundary of the supercritical regime, where $\sum_{i=1}^3 \arctan (\lambda_i)=\tfrac\pi2$, whence $1-\sumcyclic \lambda_j \lambda_k=0$.
\end{remark}

\begin{remark}[Rescaling]
	Suppose we find a solution $\kappa$ and we rescale the metric by $\omega \mapsto t \omega$. Then, we find that
	\begin{align*}
		\frac{(t\omega+ \ii \kappa)^3}{(t\omega)^3} &= \frac{(\omega+ \ii t^{-1}\kappa)^3}{\omega^3} = \prod_{i=1}^3 (1+ \ii t^{-1}\lambda_i)  \\
		&= \left(\prod_{i=1}^3 \sqrt{1+ t^{-2} \lambda_i^2 }\right) \exp \left( \ii \sum_{i=1}^3 \arctan (t^{-1}\lambda_i) \right).
	\end{align*}
	Thus
	\begin{equation}
		\Theta_{t\omega}(\kappa) = \Theta_{\omega}(t^{-1}\kappa)=\sum_{i=1}^3 \arctan(t^{-1}\lambda_i).
	\end{equation}
	This quantity can be made as close to zero as we want by making $t$ large.
\end{remark}

For $V$ being either $\bbF_2$, a hypersurface, or a curve, we now compute the central charges $Z_V(\kappa)$ (defined in Definition \ref{def:Z}) in terms of the coefficients of $\kappa$ and of an invariant representative of the Poincar\'e dual to $V$.

\begin{lem}[Central charges]\label{lem:central_charges}
	Let $C$ be a curve in $\bbF_2$ whose Poincar\'e dual is represented by $\alpha_C= \sumcyclic \mu_{i} (\frac{\ii}{2} \alpha_j \wedge \overline{\alpha_j}) \wedge (\frac{\ii}{2} \alpha_k \wedge \overline{\alpha_k})$.  
	Let $\surface$ be a hypersurface in $\bbF_2$ whose Poincar\'e dual is represented by $\alpha_\surface= \sum_{i=1}^3 \mu_{i} \frac{\ii}2 \alpha_i \wedge \overline{\alpha_i}$. Let $\kappa=\sum_{i=1}^3 \lambda_{i}  \frac{\ii}2\alpha_i \wedge \overline{\alpha_i}$ represent a class in $\Homology^{1,1}(\bbF_2,\Reals)$. We then have
	\begin{align}
		Z_C(\kappa)&=\left[ - \biggl( \sum_{i=1}^3 \mu_i \lambda_i \biggr)  + \ii \biggl( \sum_{i=1}^3 \mu_i  \biggr)  \right]\Vol(\bbF_2),\label{eq:ZC}\\
		Z_\surface(\kappa)&=\left[ \biggl( \sum_{i=1}^3 \mu_i  - \sumcyclic \mu_i \lambda_j \lambda_k \biggr) + \ii  \biggl( \sumcyclic \mu_i (\lambda_j + \lambda_k) \biggr)  \right]\Vol(\bbF_2),\label{eq:ZV}\\
		Z_{\bbF_2}(\kappa)&= \left[ \biggl( \sum_{i=1}^3\lambda_i -  \lambda_1 \lambda_2 \lambda_3 \biggr) - \ii \biggl(1- \sumcyclic \lambda_j \lambda_k\biggr)  \right]\Vol(\bbF_2).\label{eq:ZX}
	\end{align}
\end{lem}

\begin{proof}
For $\bbF_2$,
we compute that
\begin{align*}
	Z_{\bbF_2}(\kappa) & = -  \ii \frac{(\omega+\ii \kappa)^3}{\omega^3}  \int_{X} \frac{\omega^3}{3!}, \\
	& =   \left[ \biggl( \sum_{i=1}^3\lambda_i -  \lambda_1 \lambda_2 \lambda_3 \biggr) - \ii \biggl(1- \sumcyclic \lambda_j \lambda_k\biggr)  \right]  \int_{\bbF_2} \frac{ \omega^3}{3!}.
\end{align*}

For the hypersurface $\surface$, we  compute that
\[Z_{\surface}(\kappa)  = - \int_{\surface} \exp( - \ii ( \omega + \ii \kappa)) = \frac{1}{2} \int_{\surface} ( \omega + \ii \kappa)^2 = \frac{1}{2} \int_{\bbF_2} \alpha_\surface \wedge ( \omega + \ii \kappa)^2 .\]
Since
\begin{align*}
	\frac{1}{2}\alpha_\surface \wedge (\omega + \ii \kappa)^2 & = \alpha_\surface \wedge \left( \sumcyclic (1+\ii\lambda_j ) (1+\ii \lambda_k) \left( \frac{\ii}{2} \alpha_j \wedge \overline{\alpha_j} \right) \wedge \left( \frac{\ii}{2} \alpha_k \wedge \overline{\alpha_k} \right) \right) \\
	& = \sumcyclic \mu_i  \left( (1- \lambda_j \lambda_k ) + \ii (\lambda_j + \lambda_k) \right) \frac{\omega^3}{3!},
\end{align*}
we have
\begin{align*}
	Z_{\surface}(\kappa) =  \left[ \biggl( \sum_{i=1}^3 \mu_i  - \sumcyclic \mu_i \lambda_j \lambda_k \biggr) + \ii  \left( \sumcyclic \mu_i (\lambda_j + \lambda_k) \right)  \right] \int_{X} \frac{\omega^3}{3!} .
\end{align*}

For the curve $C$, we compute that 
\begin{align*}
	Z_C (\kappa) = - \int_C \exp (-\ii (\omega + \ii \kappa)) = - \int_C \kappa +\ii \int_C \omega  = - \int_X \kappa \wedge \alpha_C + \ii \int_X \omega \wedge \alpha_C .
\end{align*}
Inserting the formulas above for $\alpha_C$ and $\kappa$ we find
\begin{align*}
	Z_C(\kappa) & = - \int_X \left(\sum_{i=1}^3 \mu_i \lambda_i ) \frac{\omega^3}{3!} + \ii \int_X (\sum_{i=1}^3 \mu_i \right) \frac{\omega^3}{3!} \\
	& = \left[ - \biggl( \sum_{i=1}^3 \mu_i \lambda_i \biggr)  + \ii \biggl( \sum_{i=1}^3 \mu_i  \biggr)  \right] \int_X \frac{\omega^3}{3!} .
\end{align*}
The proof is now complete.
\end{proof}

	\subsection{Quotients of central charges}
	
	In this subsection our goal is to determine the signs of the central charge quotients. We start with the case of hypersurfaces. 

	\begin{lemma} \label{lem:ZVZXpositive}
		For $\surface$ a hypersurface, 
	 \[\Im \left(\frac{Z_{\surface}(\kappa)}{Z_{\bbF_2}(\kappa)} \right)>0.\]
	 \end{lemma}
	 \begin{proof}
		 From lemmas  \ref{lem:central_charges} and \ref{lem:PD hypersurface}, 
we know there are $\mu_1,\mu_2,\mu_3>0$ and $\lambda_1,\lambda_2,\lambda_3$ such that
	 \[
	 \frac{Z_{\surface}(\kappa)}{Z_{\bbF_2}(\kappa)}=\frac{  \bigl( \sum_{i=1}^3 \mu_i  - \sumcyclic \mu_i \lambda_j \lambda_k \bigr) + \ii  \bigl( \sumcyclic \mu_i (\lambda_j + \lambda_k) \bigr)}{\bigl( \sum_{i=1}^3\lambda_i -  \lambda_1 \lambda_2 \lambda_3 \bigr) - \ii \bigl(1- \sumcyclic \lambda_j \lambda_k\bigr)}.
	 \]
Therefore the sign of the imaginary part of $\frac{Z_{\surface}(\kappa)}{Z_{\bbF_2}(\kappa)}$ is of the same as the sign of the imaginary part of 
\[ \biggl( \sum_{i=1}^3 \mu_i  - \sumcyclic \mu_i \lambda_j \lambda_k  + \ii   \sumcyclic \mu_i (\lambda_j + \lambda_k)\biggr)\biggl( \sum_{i=1}^3\lambda_i -  \lambda_1 \lambda_2 \lambda_3  +\ii \bigl(1- \sumcyclic \lambda_j \lambda_k\bigr)\biggr).\]
This imaginary part is equal to
 \begin{align*}
\sum_{i=1}^3\lambda_i&\sumcyclic\mu_i(\lambda_j+\lambda_k)- \lambda_1 \lambda_2 \lambda_3\sumcyclic\mu_i(\lambda_j+\lambda_k)-\sumcyclic\mu_i\lambda_j\lambda_k\\
&\quad-\sumcyclic\lambda_j\lambda_k\biggl(\sum_{i=1}^3\mu_i-\sumcyclic\mu_i\lambda_j\lambda_k\biggr)+\sum_{i=1}^3\mu_i\\
=&\sumcyclic\mu_i\bigl[(\lambda_j+\lambda_k)^2- \lambda_1 \lambda_2 \lambda_3(\lambda_j+\lambda_k)-2\lambda_j\lambda_k+\lambda^2_j+\lambda^2_k\bigr]\\
&\quad+\sumcyclic\mu_i\lambda_i\lambda_j\lambda_k(\lambda_j+\lambda_k)
+\sum_{i=1}^3\mu_i\\
=&\sumcyclic\mu_i\left[\lambda^2_j+\lambda^2_k- \lambda_1 \lambda_2 \lambda_3(\lambda_j+\lambda_k)+\lambda^2_j\lambda^2_k\right]+\lambda_1 \lambda_2 \lambda_3\sum_{i=1}^3\mu_i(\lambda_j+\lambda_k)+\sum_{i=1}^3\mu_i\\
=&\sumcyclic\mu_i[\lambda^2_j+\lambda^2_k+\lambda^2_j\lambda^2_k]+\sum_{i=1}^3\mu_i.\\
 \end{align*}
 Therefore
\begin{equation}\label{eq:central charge quotient hypersurfaces}
	\sign \left[ \Im \biggl(\frac{Z_{\surface}(\kappa)}{Z_{\bbF_2}(\kappa)} \biggr) \right] = \sign \biggl[  \sum_{i=1}^3 \mu_i  + \sumcyclic \lambda_i^2 (\mu_j + \mu_k) + \sumcyclic \mu_i \lambda_j^2 \lambda_k^2 \biggr].
\end{equation}
This quantity is always positive. 
\end{proof}

\begin{lem}
	\label{lem:centralchargequotientcurves}
		Let $C$ be a curve whose Poincar\'e dual is the class of 
		$\sumcyclic \mu_i (\frac{\ii}{2} \alpha_j \wedge \overline{\alpha_j}) \wedge (\frac{\ii}{2} \alpha_k \wedge \overline{\alpha_k})$, and let  $\kappa=\sum_{i=1}^3 \lambda_i (\frac{\ii}{2} \alpha_i \wedge \overline{\alpha_i})$,
		we have
	\begin{equation}\label{eq:central charge quotient curves}
		\sign \left[ \Im \left(\frac{Z_C (\kappa)}{Z_{\bbF_2} (\kappa)}\right) \right] = \sign \left[ \sumcyclic \mu_i (1+\lambda_i^2) (\lambda_j + \lambda_k) \right] .
	\end{equation}
\end{lem}
\begin{proof}
	From Lemma \ref{lem:central_charges}, we see that
	 \[
		 \frac{Z_{C}(\kappa)}{Z_{\bbF_2}(\kappa)}=\frac{-\left( \sum_{i=1}^3 \mu_i \lambda_i \right)  + \ii \left( \sum_{i=1}^3 \mu_i  \right) }{\left( \sum_{i=1}^3\lambda_i -  \lambda_1 \lambda_2 \lambda_3 \right) - \ii \left(1- \sumcyclic \lambda_j \lambda_k\right)}.\]
	 Therefore the sign of the imaginary part of $\frac{Z_{C}(\kappa)}{Z_{\bbF_2}(\kappa)}$ is the sign of  
	 \begin{align*}
	\Im\Biggl[\biggl( -\bigl( \sum_{i=1}^3 &\mu_i \lambda_i \bigr)  + \ii \biggl( \sum_{i=1}^3 \mu_i  \biggr)\biggr)\biggl( \sum_{i=1}^3\lambda_i -  \lambda_1 \lambda_2 \lambda_3  + \ii \bigl(1- \sumcyclic \lambda_j \lambda_k\bigr)\biggr)\Biggr]\\
	 &=- \sum_{i=1}^3 \mu_i \lambda_i+ \sum_{i=1}^3 \mu_i \lambda_i\sumcyclic \lambda_j \lambda_k+ \sum_{i=1}^3 \mu_i \lambda_i+\sumcyclic \mu_i (\lambda_j+\lambda_k) -  \lambda_1 \lambda_2 \lambda_3\sum_{i=1}^3 \mu_i\\
	 &= \sumcyclic \mu_i \lambda^2_i( \lambda_j +\lambda_k)-  \lambda_1 \lambda_2 \lambda_3\sum_{i=1}^3 \mu_i+\sumcyclic \mu_i (\lambda_j+\lambda_k) -  \lambda_1 \lambda_2 \lambda_3\sum_{i=1}^3 \mu_i\\
	 &=\sumcyclic \mu_i (1+\lambda^2_i)( \lambda_j +\lambda_k).
	\end{align*}
	The claim follows.
\end{proof}

	\begin{proposition}\label{prop:sign quotients}
Let $C_1$, $C_2$, $C_3$ be the curves from Lemma \ref{lemma:C_i}. 
		Let $\beta_m= m_1 \beta_3 -  m_3 \beta_1$ for $m \in \bbZ^2$ regarded as a connection on $L_{\beta_m}$ and $F_m$ its curvature. Then
		\begin{align*}
			\sign \left[ \Im \left(\frac{Z_{C_1} (-\ii F_m)}{Z_{\bbF_2} (-\ii F_m)}\right) \right] & = - \sign \left[ m_1 + m_3 \frac{A_1^2 +2 A_3^2}{A_3^2}  \right] , \\
			\sign \left[ \Im \left(\frac{Z_{C_2} (-\ii F_m)}{Z_{\bbF_2} (-\ii F_m)}\right) \right] & = - \sign \left[ \frac{m_3}{A_3^2} +  \frac{m_1}{A_1^2} \right] , \\
			\sign \left[ \Im \left(\frac{Z_{C_3} (-\ii F_m)}{Z_{\bbF_2} (-\ii F_m)}\right) \right] & = -  \sign \left[ m_1 \frac{A_3^2+2A_1^2}{A_1^2} + m_3 \right]  .
		\end{align*}
	\end{proposition}
	\begin{proof}
		Lemma \ref{lem:PD curves} gives an expression for the Poincaré dual to $C_i$, while the curvature $F_m$ was computed in Equation \eqref{eq:curvature U(1)}. 
		Inserting the corresponding $\mu_i$ and $\lambda_j$ in Equation \eqref{eq:central charge quotient curves} and eliminating quantities that are obviously positive, we obtain the result.
	\end{proof}

	\begin{corollary}\label{cor:sign quotient}
		There are $\beta_m$ such that
		\[\sign \left[ \Im \Bigl(\frac{Z_{C_1} (-\ii F_m)}{Z_{\bbF_2} (-\ii F_m)}\Bigr) \right] = - \sign \left[ \Im \Bigl(\frac{Z_{C_3} (-\ii F_m)}{Z_{\bbF_2} (-\ii F_m)}\Bigr) \right]\]
	\end{corollary}
	\begin{proof}
		Take for instance the case where $m_3=-m_1$. Then
		\begin{align*}
			\sign \left[ \Im \Bigl(\frac{Z_{C_1} (-\ii F_m)}{Z_{\bbF_2} (-\ii F_m)}\Bigr) \right] &= 
				\sign \left[\frac{m_1(A_1^2+A_3^2)}{A_3^2}\right]=\sign(m_1),\text{ while}\\
			\sign \left[ \Im \Bigl(\frac{Z_{C_3} (-\ii F_m)}{Z_{\bbF_2} (-\ii F_m)}\Bigr) \right] &= 
				\sign \left[\frac{-m_1(A_1^2+A_3^2)}{A_1^2}\right]=-\sign(m_1).		\end{align*}
The result follows.
	\end{proof}
	
\subsection{The K\"ahler--Einstein case}
	
	In the K\"ahler--Einstein case we have in addition that $A_1=A=A_3$ and $A_2^2=2A^2$. Then, Proposition \ref{prop:sign quotients} shows that for $m \in \bbZ^2$ 
	\begin{equation}\label{eq:central charges in KE}
		\begin{aligned}
			\sign \left[ \Im \left(\frac{Z_{C_1} (-\ii F_m)}{Z_{\bbF_2} (-\ii F_m)}\right) \right] & =  - \sign \left[ m_1 + 3m_3 \right] , \\
			\sign \left[ \Im \left(\frac{Z_{C_2} (-\ii F_m)}{Z_{\bbF_2} (-\ii F_m)}\right) \right] & = - \sign \left[ m_1 + m_3 \right] , \\
			\sign \left[ \Im \left(\frac{Z_{C_3} (-\ii F_m)}{Z_{\bbF_2} (-\ii F_m)}\right) \right] & =  - \sign \left[ 3m_1 + m_3 \right]  .
		\end{aligned}
	\end{equation}
Notice that these signs do not depend on the scaling factor $A$. On the other hand, using the fact that $\lambda_i=-\frac{2m_i}{\epsilon_i A_i^2}$ for $-\ii F_m$, we have
	\begin{equation}\label{eq:lifted theta KE}
		\Theta_{\omega}(-\ii F_m) = - \arctan \left( \frac{2m_1}{A^2} \right) - \arctan \left( \frac{m_1+m_3}{A^2} \right)  - \arctan \left( \frac{2m_3}{A^2} \right)  \in \left( -\frac{3\pi}{2} , \frac{3\pi}{2} \right).
	\end{equation}

	In particular, for any $m \in \bbZ^2$ we can make $\Theta_{\omega}(-\ii F_m)$ as close to zero as we want by increasing $A$ (the scaling factor); and if $m_1m_3>0$ as close to $\pm \frac{3\pi}{2}$ as needed, by making $A$ small. Using this formula, we can completely characterise the $(1,1)$-classes, in the $m=(m_1,m_3)$ coordinates, that are supercritical up to sign meaning that $|\Theta_{\omega}(-\ii F_m)|> \frac{\pi}{2}$. Figure \ref{fig:supercritical-12} displays the supercritical regions for $A=1$ and $A=2$ for comparison.	Figure \ref{fig:samesign} shows the region for which  $\Im\bigl(\frac{Z_{C_i}(-\ii F_m)}{Z_{\bbF_2} (-\ii F_m)}\bigr)$ all have the same sign for $i=1,2,3$.
		Notice that by superposing the plots in figures \ref{fig:supercritical-12} and \ref{fig:samesign}, we find the non-supercritical $(1,1)$-classes for which the quotient of the central charges for curves is of fixed sign. For concreteness we plot the set of such classes for $A=1$ and $A=2$ in figures \ref{fig:exclusion}. The fact that this set is nonempty shows that Corollary \ref{cor:non-supercritical} is relevant.
		\begin{figure}[htbp]
		\centering
		\begin{subfigure}[t]{0.3\textwidth}
		    \centering
		    \includegraphics[width=\linewidth]{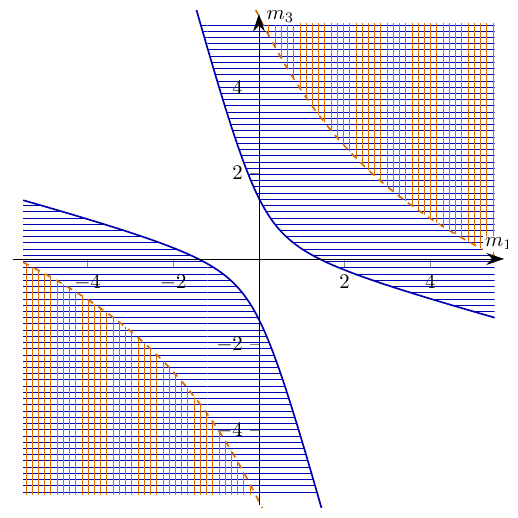}
		    \caption{The supercritical regions $S_1$ and $S_2$, corresponding to $A=1$ (horizontal hatching) and $A=2$ (vertical hatching).}\label{fig:supercritical-12}
		\end{subfigure}
		\hfill
		\begin{subfigure}[t]{0.3\textwidth}
		    \centering
		    \includegraphics[width=\linewidth]{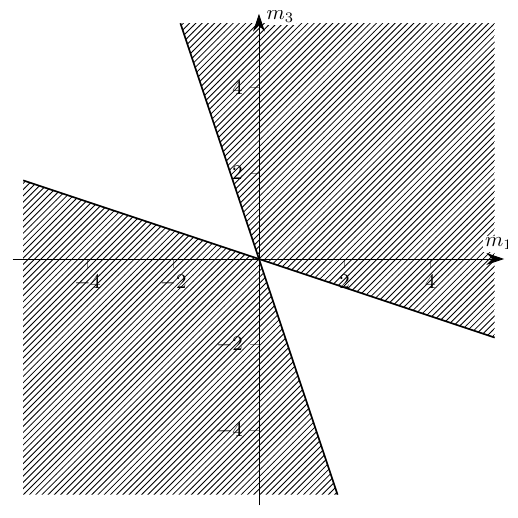}
			\caption{The same-sign region $M$, where $\Im \frac{Z_{C_i}(\kappa_m)}{Z_{\bbF_2}(\kappa_m)}$ has the same sign for $i=1,2,3$.}\label{fig:samesign}
		\end{subfigure}
		\hfill
		\begin{subfigure}[t]{0.3\textwidth}
		    \centering
		    \includegraphics[width=\linewidth]{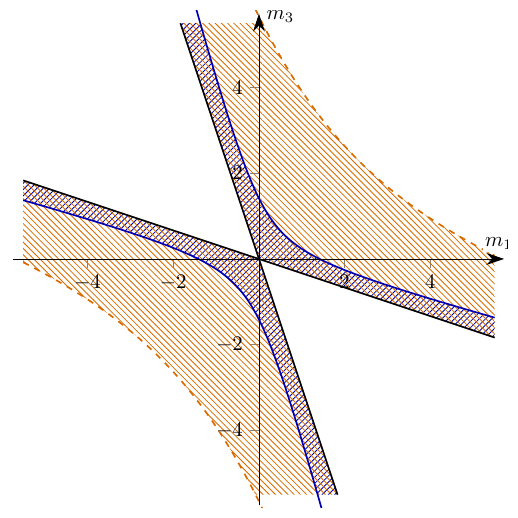}
			\caption{The exclusion regions $M\setminus S_1$ and $M\setminus S_2$. 
	The larger region $M\setminus S_2$ is shown with diagonal hatching, and the smaller region $M\setminus S_1 \subset M\setminus S_2$ by cross-hatching.}\label{fig:exclusion}
		\end{subfigure}
		\caption{Comparison of the supercritical regions and the same-sign region in the $H^2(\bbF_2,\Reals)$ plane in the  $(m_1,m_3)$ coordinates.}
		\label{fig:supercritical-samesign-exclusion}
		\end{figure}

	\begin{corollary}[Non-supercritical classes for which $\Im \frac{Z_{V}(-\ii F_m)}{Z_{\bbF_2} (-\ii F_m)}>0$ for all subvariety $V$]
		\label{cor:non-supercritical}
		Let $\bbF_2$ be equipped with a homogeneous K\"ahler--Einstein metric. Then, for any $m=(m_1,m_3) \in \bbZ^2$ such that $3m_1+m_3<0$ and $m_1+3m_3<0$, there is $A_m>0$ such that for all $A>A_m$, the connection $\beta_m$ on $L_{\beta_m}$ is not a supercritical dHYM connection, but it satisfies $\Im \frac{Z_{V}(-\ii F_m)}{Z_{\bbF_2} (-\ii F_m)}>0$ for all analytic subvarieties $V \subset \bbF_2$. 
	\end{corollary}
	\begin{proof}
		Lemma \ref{lem:ZVZXpositive} shows that $\Im \frac{Z_{\surface}(-\ii F_m)}{Z_{\bbF_2} (-\ii F_m)}>0$ for  any hypersurface $\surface\subset \bbF_2$. Now we turn to the case in which $V=C$ is a curve. First, we point out that it is enough to impose that $\Im \frac{Z_{V}(-\ii F_m)}{Z_{\bbF_2} (-\ii F_m)}>0$ for $V=C_1$ and $V=C_3$. Indeed, given that $PD[C_1]$ and $PD[C_3]$ generate $\Homology^4(\bbF_2 , \bbZ)$, any other curve represents a class whose Poincar\'e dual is of the form $n_1 PD[C_1] + n_3 PD[C_3]$ with $n_1,n_3 \geq 0$. This follows from the fact that any holomorphic curve is calibrated with respect to any K\"ahler metric $\omega=\sum_{i=1}^3 \tfrac{\ii}{2} \alpha_i \wedge \overline{\alpha_i}$. Then, we must have that $$0<\int_{C}\omega = n_1 \int_{C_1}\omega + n_3 \int_{C_3} \omega = n_1 A_1^2 + n_3 A_3^2,$$ and for this to be positive for all $A_1,A_3>0$, we must have $n_1,n_3$ are non-negative.
		Hence, $Z_C= n_1 Z_{C_1}+ n_3Z_{C_3}$ and so $\Im \frac{Z_{C}(-\ii F_m)}{Z_{\bbF_2} (-\ii F_m)}$ is positive if both $\Im \frac{Z_{C_1}(-\ii F_m)}{Z_{\bbF_2} (-\ii F_m)}$ and $\Im \frac{Z_{C_3}(-\ii F_m)}{Z_{\bbF_2} (-\ii F_m)}$ are. Making use of Equations \eqref{eq:central charges in KE}, this positivity is ensured by the hypothesis that $3m_1+m_2<0$ and $m_1+3m_2<0$. Hence, we are left with proving that it is possible to increase $A$ so that $\beta_m$ is not supercritical. This follows from Equation \eqref{eq:lifted theta KE} and the fact the expression in the right hand side converges to zero as $A \to + \infty$ independently of $m$.
	\end{proof}

	\begin{remark}
		Again, one can consider real $(1,1)$-classes by letting $m \in \Reals^2$ instead of $\bbZ^2$. Then, we plot in figures \ref{fig:exclusion} the classes for which the signs of $\Im \frac{Z_{C_i}(\kappa_m)}{Z_{\bbF_2} (\kappa_m)}$ all coincide, but which are not supercritical. Recall that all classes admit solutions to the dHYM equation.
	\end{remark}

	Finally, figure \ref{fig:different signs} displays the region for which at least two of the $\Im \frac{Z_{C_i}(-\ii F_m)}{Z_{\bbF_2} (-\ii F_m)}$ have different signs. This region is obtained from the following result.
	
	\begin{corollary}[\citep{Theorem 1}{Jacob-Sheu-dHYM-blowup-Pn} does not generalize]
		Let $\bbF_2$ be equipped with the homogeneous K\"ahler--Einstein metric and $m \in \bbZ^2$. The following hold:
		\begin{itemize}
			\item if $m_1+3m_3<0<3m_1+m_3$, then $\Im \left( \frac{Z_{C_1}(-\ii F_m)}{Z_{\bbF_2} (-\ii F_m)} \right) >0>\Im \left( \frac{Z_{C_3}(-\ii F_m)}{Z_{\bbF_2} (-\ii F_m)} \right)$;
			\item if $m_1+3m_3>0>3m_1+m_3$, then $\Im \left( \frac{Z_{C_1}(-\ii F_m)}{Z_{\bbF_2} (-\ii F_m)} \right) <0<\Im \left( \frac{Z_{C_3}(-\ii F_m)}{Z_{\bbF_2} (-\ii F_m)} \right)$.
		\end{itemize}
		 Furthermore, in either situation, the connection $\beta_m$ on $L_{\beta_m}$ satisfies the dHYM connection.
	\end{corollary}

	\begin{figure}[h]
		\centering
		\includegraphics[scale=0.4]{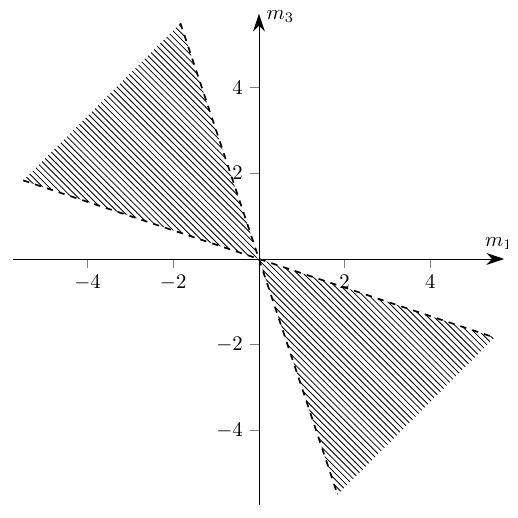}
		\caption{The $(1,1)$-classes for which at least two of the $\Im \frac{Z_{C_i}(-\ii F_m)}{Z_{\bbF_2} (-\ii F_m)}$ have different signs.}\label{fig:different signs}
	\end{figure}

\begin{cor}
	Let $m_1+m_3=0$, then 
	\[\Im \Biggl( \frac{Z_{C_2}(-\ii F_m)}{Z_{\bbF_2} (-\ii F_m)} \Biggr)=0,\]
	and there is a solution to the dHYM equation in the class $[-\ii F_m]$.
\end{cor}

\subsection{The hypercritical regime}

Let us start by characterising the classes $[\kappa]$ for which $\Im Z_{\bbF_2}(\kappa)>0$.  

\begin{lem}\label{lem:ZX>0}
	Let $(m_1,m_3) \in \Reals^2$ and the corresponding class $\kappa_{m}=-\ii F_m \in \Homology^{1,1}(\bbF_2, \Reals)$. Then, $\Im Z_{\bbF_2}(\kappa_m)>0$ for the K\"ahler--Einstein structure with scaling $A>0$ if and only if
	\[(m_1+m_3)^2+2m_1m_3 > \frac{A^4}{2}.\]
\end{lem}
\begin{proof}
	Using Equation \eqref{eq:ZX}, we find that these classes consists of those for which
	\[1<\sum_{i=1}^3 \lambda_j \lambda_k .\]
	 Now, recall that $\lambda_i = - \frac{2m_i}{\epsilon_i A_i^2}$ for $i=1,2,3$, while $\epsilon_1=1=\epsilon_3$ and $\epsilon_2=-1$. Furthermore, $m_2=-(m_1+m_3)$, and in the K\"ahler--Einstein case $A_1^2=A^2=A_3^2$ and $A_2^2=2A^2$. Thus, the condition becomes
	 \begin{align*}
	 	1&< \frac{2m_1}{A^2} \frac{m_1+m_3}{A^2} + \frac{m_1+m_3}{A^2}\frac{2m_3}{A^2} + \frac{2m_3}{A^2} \frac{2m_1}{A^2}\\
		 &= \frac{2}{A^4}\Bigl(m_1(m_1+m_3)+m_3(m_1+m_3)+2m_1m_3\Bigr)\\
		 &= \frac{2}{A^4}\Bigl((m_1+m_3)^2+2m_1m_3\Bigr).
	 \end{align*}
The result follows.
\end{proof}

A representation of set of pairs $(m_1,m_3)$ for which the inequality $(m_1+m_3)^2+2m_1m_3 > \frac{A^4}{2}$ of Lemma \ref{lem:ZX>0} holds is given in figure \ref{fig:ZX>0} for $A=1$.

\begin{figure}[h]
	\centering
	\includegraphics[scale=0.4]{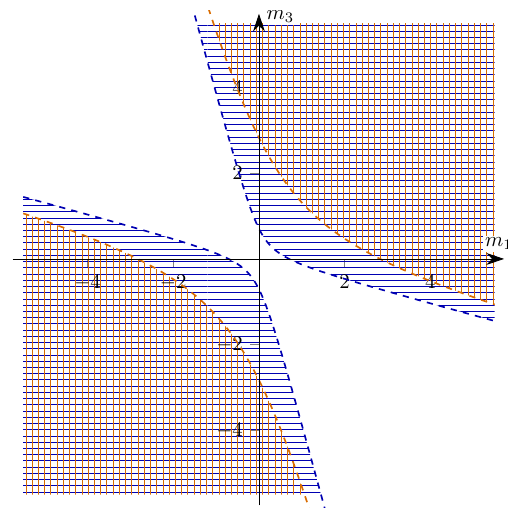}
	\caption{The classes which satisfy $\Im Z_{\bbF_2}>0$ for the K\"ahler--Einstein structure with scaling $A=1$ (horizontal hatching) and $A=2$ (vertical hatching).}\label{fig:ZX>0}
\end{figure}

Now we characterise those classes for which $\Im Z_V >0$ for all irreducible sub-varieties $V$.

\begin{lem}\label{lem:ZV>0}
	Let $(m_1,m_3) \in \Reals^2$ and consider the corresponding class $\kappa_{m} \in \Homology^{1,1}(\bbF_2, \Reals)$. Then, considering the K\"ahler--Einstein structure with scaling $A>0$, we have $\Im Z_V(\kappa_m)>0$ for all irreducible subvarieties $V \subsetneq X$ if and only if
	\[m_1+2m_3 < 0 , \quad 2m_1+m_3<0.\]
\end{lem}
\begin{proof}
Suppose first that $V=C$, a curve. Then using the definition \ref{def:Z} of central charge, we see that $\Im Z_C(\kappa_m)=\int_C \omega = \Area(C) >0$.

Suppose now that $V=\surface$, a hypersurface. Recall from Lemma \ref{lem:PD hypersurface} that we must have $\mu_2=\tfrac{A_1^2\mu_1+A_3^2\mu_3}{A_2^2}=\tfrac{\mu_1+\mu_3}2$, and that $\mu_1,\mu_3\geq0$. Using Equation \eqref{eq:ZV}, we see that 
\begin{align*}
	\Im Z_\surface(\kappa_m)&=\Vol(\bbF_2)\sumcyclic \mu_i(\lambda_j+\lambda_k)\\
	 &=\Vol(\bbF_2)\bigl(\mu_1\frac{m_1-m_3}{A^2}+\frac{(\mu_1+\mu_3)}2\frac{(-2m_1-2m_3)}{A^2}+\mu_3\frac{m_3-m_1}{A^2}\bigr)\\
	 &=-\frac{\Vol(\bbF_2)}{A^2}\bigl((m_1+2m_3)\mu_1+(2m_1+m_3)\mu_3\bigr).
\end{align*}
This quantity is positive for all surfaces if and only if the coefficients of $\mu_1$ and $\mu_3$ are negative.	
\end{proof}

The set determined by the two inequalities in Lemma \ref{lem:ZV>0} is represented in figure \ref{fig:ZV>0}.

Now we consider the hypercritical case, determined by the condition $\Theta(\kappa_m)= \sum_{i=1}^3 \arctan(\lambda_i) \in ( \pi ,  \tfrac{3\pi}{2} )$. For the K\"ahler--Einstein metric with scaling $A>0$, the hypercriticality condition becomes
\[  \Theta(\kappa_m)=\arctan \left( \frac{2m_1}{A^2} \right) + \arctan \left( \frac{m_1+m_3}{A^2} \right) + \arctan \left( \frac{2m_3}{A^2} \right) <-\pi . \]
The classes  satisfying this condition are portrayed in figure \ref{fig:hypercritical}.

Given that $\arctan$ takes values in $(-\tfrac{\pi}{2} , \tfrac{\pi}{2})$, we have that
\begin{align*}
	-\pi&>\arctan \left( \frac{2m_1}{A^2} \right) + \arctan \left( \frac{m_1+m_3}{A^2} \right) + \arctan \left( \frac{2m_3}{A^2} \right) \\
	  &>\arctan \left( \frac{2m_1}{A^2} \right) + \arctan \left( \frac{2m_3}{A^2} \right) - \frac{\pi}{2}
\end{align*}
implies
\begin{align*}
	& \arctan \left( \frac{2m_1}{A^2} \right) + \arctan \left( \frac{2m_3}{A^2} \right)   < - \frac{\pi}{2} ,
\end{align*}
which in turn implies that both $m_1$ and $m_3$ must be negative when $\kappa_m$ is hypercritical. In particular, as the set where both $m_1$ and $m_3$ is negative is a subset of the set characterised in Lemma \ref{lem:ZV>0}, we have that if $\kappa_m$ is hypercritical, then $\Im Z_V>0$ for all $V \subsetneq X$. The opposite, which is  \cite[Conjecture 8.5]{Collins-Yau-MomentMaps-arxiv}, is however not true as we shall now see.

\begin{cor}[Counterexample to \citep{Conjecture 8.5}{Collins-Yau-MomentMaps-arxiv}]
	\label{cor:nothypercritical_but_ZV>0}
	Let $(m_1,m_3) \in \bbZ^2$ be such that either $m_1$ or $m_3$ is non-negative, but $m_1+2m_3 < 0$ and $2m_1+m_3<0$. Then, the corresponding class $\kappa_{m}=-F_m \in \Homology^{1,1}(\bbF_2, \bbZ)$ satisfies $\Im Z_V(\kappa_m)>0$ for all $V \subsetneq X$ but it is not hypercritical.
\end{cor}
\begin{proof}
	It follows from Lemma \ref{lem:ZV>0} that if $m_1+2m_3 < 0$ and $2m_1+m_3<0$ then $\Im Z_V(\kappa_m)>0$ for all $V \subsetneq X$. However, if either $m_1$ or $m_3$ is non-negative we find, from the discussion preceding the statement, that the class $\kappa_m=-\ii F_m$ is not hypercritical.
\end{proof}

Figure \ref{fig:nothypercritical} gives a graphical representation of the counterexamples to  \cite[Conjecture 8.5]{Collins-Yau-MomentMaps-arxiv}.

		\begin{figure}[htbp]
		\centering
		\begin{subfigure}[t]{0.3\textwidth}
		    \centering
		    \includegraphics[width=\linewidth]{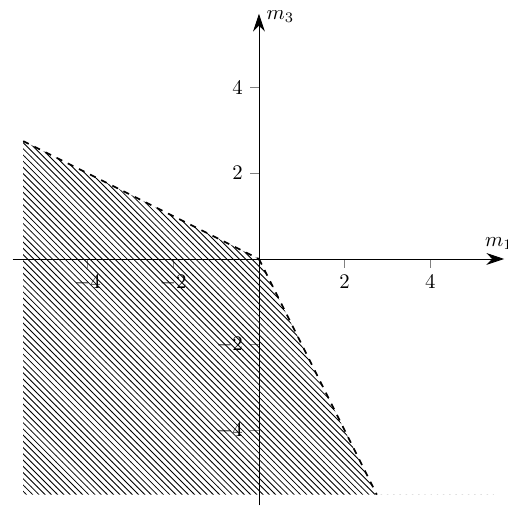}
			\caption{The positive-sign region $P$, defined by $m_1+2m_3<0$ and $2m_1+m_3<0$, equivalently the region where $\Im Z_V>0$ for all subvarieties $V$.}
			\label{fig:ZV>0}
		\end{subfigure}
		\hfill
		\begin{subfigure}[t]{0.3\textwidth}
		    \centering
		    \includegraphics[width=\linewidth]{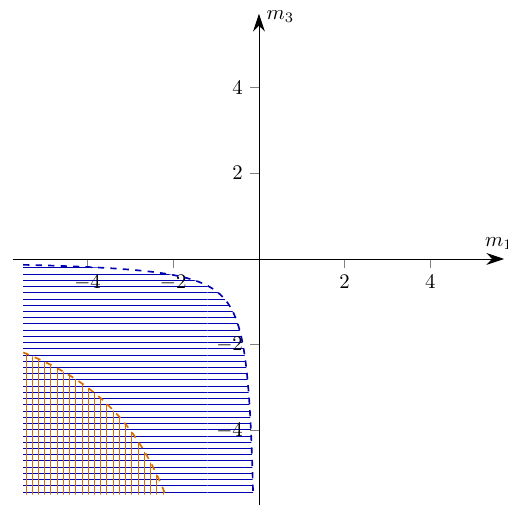}
			\caption{The hypercritical regions $H_1$ and $H_2$ for the K\"ahler--Einstein structure with scaling parameter $A=1$ (horizontal hatching) and $A=2$ (vertical hatching).}\label{fig:hypercritical}
		\end{subfigure}
		\hfill
		\begin{subfigure}[t]{0.3\textwidth}
		    \centering
		    \includegraphics[width=\linewidth]{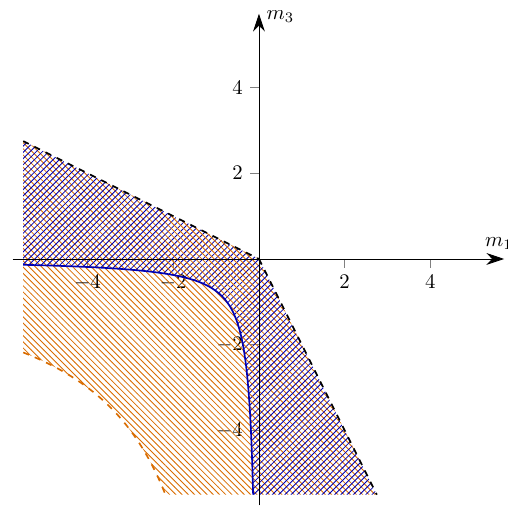}
	\caption{The classes which satisfy $\Im Z_V(\kappa_m)>0$ for all $V \subsetneq X$ but are not hypercritical. The larger diagonally hatched region is $P\setminus H_2$, while the cross-hatched nested region is $P\setminus H_1$.}\label{fig:nothypercritical}
		\end{subfigure}
\caption{The classes satisfying $\Im Z_V(\kappa_m)>0$ for all $V \subsetneq X$ but not hypercritical. The larger diagonally hatched region is $P\setminus H_2$, while the cross-hatched nested region is $P\setminus H_1$.}
		\label{fig:hypercritical-positive-exclusion}
		\end{figure}

\begin{remark}
	We point out that Chu--Lee has already provided a counterexample to this conjecture in \cite{Chu-Lee-hypercritical-dHYM-revisited}. However, their example in $Bl_p(\cp^2)$ is complex two dimensional while the conjecture had been formulated for the three dimensional case. 
\end{remark}

\section{dHYM equations in rank two}\label{2dHYM}

\subsection{The dHYM equation in higher rank}

Let $E$ be a rank $r$ holomorphic vector bundle over a compact K\"ahler manifold $(X^n,\omega)$. Then, the curvature $F_A$ of a Hermitian connection $A$ compatible with the holomorphic structure is a section of $\Lambda^{1,1}_\bbC \otimes \mathfrak{u}(E) \subset \Lambda^{1,1}_\bbC \otimes \End(E)$. 

Following section 8.1 in \cite{Collins-Yau-MomentMaps-arxiv}, see also \cite{Dervan-Zcritical-Bridgeland}, the dHYM equation in higher rank can be written as the following equation on $\Lambda^{2n}_\bbC \otimes \End(E)$
\begin{equation}\label{eq:dHYM higher rank}
	\Im \left( e^{-\ii\theta}(\omega \otimes \id_E + F_A)^{n} \right)=0,
\end{equation}
for a constant $\theta \in \Reals/\pi \bbZ$, which is again determined topologically by taking the trace of both sides and integrating over $X$.

\subsection{Homogeneous $\U(n)$-bundles and connections}

A homogeneous vector bundle over $\bbF_2$ is a bundle whose total space admits a lift of the $\SU(3)$ action on $\bbF_2$. An example of such a bundle is $SU(3)\rightarrow SU(3)/\bbT^2$. It is know that any homogeneous $\U(n)$-bundle over $\bbF_2=\SU(3)/\bbT^2$ can be described through the associated bundle construction and the use of an isotropy or intertwining homomorphism 
$\bbT^2\to \U(n)$. 
Any such homomorphism can be diagonalized to a 
\[\lambda_{\rho} = \diag( e^{\ii\rho_1} , \ldots ,  e^{\ii \rho_n} )  ,\]
for $\ii\rho_1 , \ldots , \ii\rho_n \in \ii (\ttt^2)^*$ integral weights of $\SU(3)$. 
The associated bundle is 
\[P_{\rho}= \SU(3) \times_{\lambda_{\rho}} \U(n).\]

\begin{remark}
	Let $V_{\rho}=P_{\rho} \times_{\U(n)} \bbC^n$ be the vector bundle associated to $P_\rho$ with respect to the standard representation of $\U(n)$. This bundle splits as a direct sum
	\[V_\rho= L_{\rho_1} \oplus \cdots \oplus L_{\rho_n}. \]
\end{remark}

Recall that  $\Homology^2(\bbF_2,\bbZ)$ is parametrized by $\beta_m=m_1 \beta_3 - m_3 \beta_1$ (introduced in Equation \eqref{eq:betam})
for $m =(m_1,m_3) \in \bbZ_2$ and that the roots of $\SU(3)$ are given by
\begin{equation}\label{eq:rootsbetam}	
\begin{aligned}
	r_1 & = \ii (\beta_2-\beta_3) = \ii (-\beta_1-2\beta_3) =\ii\beta_{(-2,1)},\\
	r_2 & = \ii (\beta_3-\beta_1) =\ii\beta_{(1,1)}, \\
	r_3 & = \ii (\beta_1-\beta_2) = \ii (2 \beta_1 + \beta_3)=\ii\beta_{(1,-2)}.
\end{aligned}
\end{equation}
It is useful to recall the notation introduced in Equations \eqref{eq:DEFmatrices} to label elements of $\mathfrak{u}(n)$.

Using the splitting of $\su(3) = \ttt^2 \oplus \mm$ with $\mm = \mm_1 \oplus \mm_2 \oplus \mm_3$ as before, we find that the canonical invariant connection on $P_\rho$ is given by the left-invariant $1$-form
\[A_\rho = \sum_{i=1}^n \rho_i D_i \in \Omega^1(\SU(3), \mathfrak{u}(n)).\]

Other invariant connections on $P_\rho$ are classified by Wang's theorem, \cite[Corollary 3]{Wang-invariant-connections}.
\begin{thm}[Wang]\label{thm:Wang}
Let $P=K\times_{H,\lambda}G$ be a principal homogeneous $G$-bundle. Let $\fk=\fh\oplus \mm$ with $[\fh,\mm]\subset\mm$.  Then $K$-invariant connections $A$ on $P$ are in one to one correspondence with linear maps $\Lambda\colon \mm \to\fg$ such that $\Lambda\circ \ad(h)=\ad(\lambda(h))\circ \Lambda$ for all $h\in\fh$.
\end{thm}
See \cite{Wang-invariant-connections} or \cite[Theorem 3.5]{Oliveira-facts-invariant-connections} for more details. 
By applying the above when $K=\SU(3)$, $H=\bbT^2$ and $G=\U(n)$, it follows that an invariant connection on $P_\rho$ is given by a connection 1-form differing from $A_\rho$ by the addition of the left-invariant extension to $\SU(3)$ of
\[0 \oplus \Lambda \colon \ttt^2 \oplus \mm \to \mathfrak{u}(n),\]
for a morphism of $\bbT^2$-representations
\begin{equation}\label{eq:Lambda}
	\Lambda \colon (\mm, \Ad_{\SU(3)}) \to (\mathfrak{u}(n), \Ad_{\U(n)} \circ \lambda_\rho).
\end{equation}
These morphisms are characterised by Schur's lemma and we must therefore decompose the above representations into irreducible components. The left hand side decomposes as the root spaces $\mm_1 \oplus \mm_2 \oplus \mm_3$ and so we decompose the right hand side. Using the notation of Equations \eqref{eq:DEFmatrices}, we have
\[\mathfrak{u}(n) \cong  \Reals^n \oplus \bigoplus_{i<j} \Reals \langle E_{ij}, F_{ij} \rangle ,\]
as $\bbT^2$-representations with the $ \Reals^n$ corresponding to a direct sum of  trivial representations and $\Reals \langle E_{ij}, F_{ij} \rangle \cong \bbC_{\ii (\rho_i - \rho_j)}$. On the other hand, we have $\mm_l \cong \bbC_{r_l}$ for $l=1,2,3$. Hence, Schur's lemma guarantees that the components of $\Lambda$ mapping each $\mm_l$ to $\Reals^n$ must vanish; and the components of $\Lambda$ mapping $\mm_l$ to $\Reals \langle E_{ij}, F_{ij} \rangle$ must vanish if $\ii(\rho_i - \rho_l) \neq \pm  r_l$. 

\begin{example}\label{ex:Tangent_Bundle}
Consider the gauge group $\SU(3)$ with the $\bbT^2$-action we have been considering throughout. This action extends to $\U(3)$. In the notation above this corresponds to having $n=3$ and $\rho_l = \beta_l$ for $l=1,2,3$. The corresponding $\U(3)$-bundle reduces to $\SU(3)$ and it is more convenient to consider instead this $\SU(3)$-subbundle. Let $\tilde{\lambda}_\rho\colon \bbT^2\rightarrow \SU(3)$ be the restriction of the action $\lambda_\rho$, and
\[Q:= \SU(3) \times_{\tilde{\lambda}_\rho} \SU(3),\]
 In fact, the corresponding associated vector bundle $V_\rho$ is the tangent bundle to $\SU(3)/\bbT^2$.
Then, the most general invariant connection of $Q$ can be written as
\[A= \sum_{i=1}^3 \beta_i \otimes T_i + \sumcyclic a_i \Bigl[ \bigl( \cos(b_i) \eta_i + \sin(b_i) \theta_i \bigr) \otimes E_{jk} + \bigl( \cos(b_i) \theta_i - \sin(b_i) \eta_i \bigr) \otimes F_{jk} \Bigr] . \] 
We shall not pursue this example here because computations indicate there are no non-trivial invariant dHYM connections on this bundle. 
\end{example}

We consider $\U(2)$-bundles. Fix $\{\TT_0,\TT_1,\TT_2,\TT_3\}$, a basis of $\mathfrak{u}(2)$ given by 
\[\TT_0= \diag (\ii,\ii), \ \TT_1 = \diag(\ii,-\ii), \ \TT_2=E_{12}, \ \TT_3=F_{12}.\]
We have $[\TT_i,\TT_j] = 2 \TT_k$ for $(i,j,k)$ a cyclic permutation of $(1,2,3)$. 
For any pair $\ii\rho=(\ii\rho_1,\ii\rho_2)$ of integral weights of $\SU(3)$ and representation morphism $\Lambda$ as per Equation \eqref{eq:Lambda}, we have a homogeneous connection on the homogeneous principal $\U(2)$-bundle $P_\rho$ over $\bbF_2$ via the connection 1-form
\begin{align*}
	A & = \rho_1 \otimes \begin{pmatrix} \ii & 0 \\ 0 & 0\end{pmatrix} + \rho_2 \otimes \begin{pmatrix} 0 & 0 \\ 0 & \ii \end{pmatrix} + \Lambda \\
	& = \frac{\rho_1 + \rho_2}{2} \otimes \TT_0 +  \frac{\rho_1 - \rho_2}{2} \otimes \TT_1 + \Lambda ,
\end{align*}
Given that $\mathfrak{u}(2) = \mathfrak{u}(1)^2 \oplus \langle \TT_2, \TT_3 \rangle$ and $\langle \TT_2, \TT_3 \rangle \cong \bbC_{\ii(\rho_1-\rho_2)}$ as representations of $\ttt^2$, the morphism $\Lambda$ can be non-zero if and only if $\ii(\rho_1-\rho_2)= \pm r_i$ for some $i \in \lbrace 1,2,3 \rbrace$. Schur's lemma implies that a nonzero morphisms of (real) representations $\psi_{\pm i} \colon \bbC_{\ii r_i} \to \langle \TT_2, \TT_3 \rangle$ exists if and only if $\ii(\rho_1-\rho_2) = \pm r_i$ in which case it is an isomorphism. Furthermore, any two such isomorphisms differ by composition with a rotation of either the target or the domain, and by multiplication by a constant. Of course as real representations, $\bbC_{\ii r_i}\cong \bbC_{-\ii r_i}$, but our choice of $\psi_{+i}$ and $\psi_{-i}$ fixes the orientation.

Now, recall that a gauge transformation is a section of $\SU(3)\times_{c \circ \lambda} \U(2)$ with $c$ the action of $\U(2)$ on itself by conjugation. Furthermore, an invariant gauge transformation is given by multiplication by a constant $g \in \U(2)$. For such a multiplication to induce a well defined gauge transformation of the bundle, it must commute with the action of the stabiliser, $\bbT^2$ by conjugation. Hence, it must take values in the maximal torus. Such gauge transformations act on $A$ by rotating $\Lambda$. Hence, up to an invariant gauge transformation, any invariant connection can be written as above with
\[\Lambda= \sum_{i=1}^3  a_i  \psi_i + a_{-i}  \psi_{-i} ,\]
for some fixed isomorphisms $\psi_{\pm i}$ and $a_{\pm i} \in \Reals$. Furthermore, recall for the discussion above that $a_{\pm i}$ can only be nonzero if $\ii(\rho_1-\rho_2)= \pm r_i$ and so there can be only at most one nonzero $a_{\pm i}$ for any fixed $\Lambda$.

\begin{remk}
	Let $m=(m_1,m_3), p=(p_1,p_3) \in \bbZ^2$ so that $\rho_1-\rho_2=\beta_m$ and $\rho_1+\rho_2=\beta_p$. From
	\[\rho_1-\rho_2= m_1 \beta_3 - m_3 \beta_1, \ \ \ \rho_1+\rho_2=p_1 \beta_3 - p_3 \beta_1,\]
	we find that
	\begin{align*}
		\rho_1 & = \frac{p_1+m_1}{2} \beta_3 - \frac{p_3+m_3}{2} \beta_1, \\
		\rho_2 & = \frac{p_1-m_1}{2} \beta_3 - \frac{p_3-m_3}{2} \beta_1 .
	\end{align*}
	In particular, for these to be integral weights we must have that $p_i\pm m_i \in 2 \bbZ$ for $i \in \{ 1 , 3\}$. We shall now analyse these condition in the cases in which $i(\rho_1-\rho_2)=i\beta_m= \pm r_j$ for $j \in \{ 1,2,3 \}$.
	
	\begin{itemize}
		\item When $\ii\beta_m=\ii(\rho_1-\rho_2)= \pm r_1=\pm \ii\beta_{(-2,1)}$, we have that $p_1$ is even and $p_3$ is odd.
		
		\item When $\ii\beta_m=\ii(\rho_1-\rho_2)= \pm r_2=\pm\ii\beta_{(1,1)}$,  we have that both $p_1$ and $p_3$ are odd.
		
		\item When $\ii\beta_m=\ii(\rho_1-\rho_2)= \pm r_3=\pm\ii\beta_{(1,-2)}$, we have that $p_1$ odd and $p_3$ even.
	\end{itemize}
\end{remk}

 We summarise our conclusions in the following result.

\begin{lem}\label{lem:gaugedconnection}
	Let $A$ be an invariant connection on the $\U(2)$-bundle $P_\rho$ from above. Let $\psi_{\pm i}\colon \mm_i \to \langle T_2, T_3 \rangle$ be fixed isomorphisms. Then, up to the use of an invariant gauge transformation, there are $a_{\pm i } \in \Reals$ such that 
	\begin{equation}
	A=\frac{\beta_p}{2} \otimes \TT_0 + \frac{\beta_m}{2} \otimes \TT_1 + \sum_{i=1}^3 \left( a_i  \psi_i + a_{-i} \psi_{-i} \right).
	\end{equation}
 Furthermore, $a_{\pm i}$ must vanish if $\beta_m \neq \mp ir_i$, and 
 \begin{itemize}
 	\item if $a_{\pm1}\neq 0$, then $0=a_{\mp1}=a_{+2}=a_{-2}=a_{+3}=a_{-3}$, $m=\pm(-2,1)$ and $p\in\Zeven\times\Zodd$,
	\item if  $a_{\pm2}\neq 0$, then $0=a_{\mp2}=a_{+3}=a_{-3}=a_{+1}=a_{-1}$, $m=\pm(1,1)$ and $p\in\Zodd\times\Zodd$,
	\item if  $a_{\pm3}\neq 0$, then $0=a_{\mp3}=a_{+1}=a_{-1}=a_{+2}=a_{-2}$,  $m=\pm(1,-2)$ and $p\in\Zodd\times\Zeven$.
 \end{itemize}
\end{lem} 

The curvature $F_A$ of such a connection can be computed using the Maurer--Cartan equations and the formula $F_A = dA + \frac{1}{2} [A \wedge A]$. To compute this it is helpful to write 
\[A= \sum_{\mu=0}^3 \mA_\mu \otimes \TT_\mu , \text{ and }F_A= \sum_{\mu=0}^3 F_\mu \otimes \TT_\mu.\]
We have
\begin{align*}
	F_A & = \sum_{\mu=0}^3 d\mA_\mu \otimes \TT_\mu + \frac{1}{2} \left[ \left( \mA_0 \otimes \TT_0 + \sum_{i=1}^3 \mA_i \otimes \TT_i \right) \wedge \left( \mA_0 \otimes \TT_0 + \sum_{i=1}^3 \mA_i \otimes \TT_i \right) \right] \\
	& = \sum_{\mu=0}^3 d\mA_\mu \otimes \TT_\mu + \frac{1}{2}  \sum_{i,j=1}^3 \mA_i \wedge \mA_j \otimes [\TT_i,\TT_j]  \\
	& = d\mA_0 \otimes \TT_0 + \sumcyclic (d\mA_i + 2 \mA_j \wedge \mA_k) \otimes \TT_i.
\end{align*}
 Summarising, we find that
\begin{equation}\label{eq:componentscurvature}
	\begin{aligned}
		F_0 & =d\mA_0 \\
		F_1 & =d\mA_1 + 2 \mA_2 \wedge \mA_3,\\
		F_2 & =d\mA_2 + 2 \mA_3 \wedge \mA_1,\\
		F_3 & =d\mA_3 + 2 \mA_1 \wedge \mA_2.
	\end{aligned}
\end{equation} 

\begin{remk}\label{rem:curvaturedescends}
By construction, $F_A$ descends to $\bbF_2$. But one has to be careful: the $F_i$ do not necessarily descend themselves. One must indeed consider the action on $\mathfrak{u}(2)$ as well. When checking that a regular form $\nu$ descends, we only have to verify that $\cL_X(\nu)=0$ for every left-invariant vector-field $X\in \ttt^2$. For $\nu\in\Omega^k(\SU(3),\mathfrak{u}(2))$ to descend to an element of $\Omega^k(\bbF_2,P_\rho)$, we must have $\cL_X(\nu)+\tfrac{\beta_m(X)}2[\TT_1,\nu]=0$. Since $F_A$ descends, we must then have
\begin{equation}
	\cL_X(F_0)=\cL_X(F_1)=0,\quad \cL_X(F_2)-\beta_m(X)F_3=0, \text{ and }\cL_X(F_3)+\beta_m(X)F_2=0.
\end{equation}
\end{remk}

\begin{remk}\label{rem:Chern_Clases}
	From Chern--Weyl theory we know that 
	\begin{align*}
		c_1(V_\rho) = \frac{i}{2\pi} [\tr (F_A)] = \frac{i}{2\pi} [\tr (F_0\otimes\TT_0)] = - \frac{1}{2\pi} [d(\rho_1+\rho_2)] = - \frac{1}{2\pi} [d\beta_p].
	\end{align*}
	Alternatively, we can write 
	\[V_\rho= L_{\rho_1} \oplus L_{\rho_2},\]
	with $c_1(L_{\rho_i})= - \frac{1}{2\pi} [d\rho_i]$ for $i\in \{1,2\}$. In particular, we find that
	\begin{align*}
		c_1(L_{\rho_1}) + c_1(L_{\rho_2}) & = c_1(V_\rho) = - \frac{1}{2\pi} [d \rho_1 ] - \frac{1}{2\pi} [d \rho_2 ] =  - \frac{1}{2\pi} [d\beta_p] , \\
		c_1(L_{\rho_1}) - c_1(L_{\rho_2}) & = - \frac{1}{2\pi} [d (\rho_1-\rho_2) ] = - \frac{1}{2\pi} [d\beta_m].
	\end{align*}
\end{remk}

To state the next result recall that the slope of a vector bundle $V$ is defined as
\begin{equation}
	\mu(V)=\frac{\deg(V)}{\rank V}= \frac{1}{\rank V} \int_{\bbF_2} c_1(V) \wedge \omega^2.
\end{equation}

\begin{lem}\label{lem:Degrees}
	Let $V_\rho=L_{\rho_1} \oplus L_{\rho_2}$ as before, that is $\rho_1+\rho_2=\beta_p$ and $\rho_1-\rho_2=\beta_m$. Then
	\begin{align*}
		c_1(V_\rho) \wedge \omega^2 & = \frac{1}{2\pi} \frac{4}{A_1^2A_2^2A_3^2} \biggl( \sumcyclic \epsilon_i p_i A_j^2 A_k^2 \biggr) \frac{\omega^3}{3!} \\
		\bigl( c_1(L_{\rho_1}) - c_1(L_{\rho_2}) \bigr) \wedge \omega^2 & = \frac{1}{2\pi} \frac{4}{A_1^2A_2^2A_3^2} \biggl( \sumcyclic \epsilon_i m_i A_j^2 A_k^2 \biggr) \frac{\omega^3}{3!}.
	\end{align*}
	In particular, let $v:= \frac{2}{A_1^2A_2^2A_3^2} \int_X\frac{\omega^3}{3!}>0$ which is a scale invariant quantity. Then,
	\begin{align*}
		\mu(L_{\rho_1})- \mu(V_\rho) & = \frac{v}{2\pi} \sumcyclic \epsilon_i m_i A_j^2 A_k^2 , \\
		\mu(L_{\rho_2})- \mu(V_\rho) & = - \frac{v}{2\pi} \sumcyclic \epsilon_i m_i A_j^2 A_k^2 .
	\end{align*}
\end{lem}
\begin{proof}
	Let $\omega_i:=\frac{\ii}{2} \alpha_i \wedge \overline{\alpha_i}$ so that $\omega = \omega_1 + \omega_2 + \omega_3$. Then, using $p_2=-p_1-p_3$, 
	we find that
	\begin{align*}
		\omega^2 \wedge d\beta_p 
		& = -4 \bigl( \sum_{i=1}^3 \frac{p_i}{\epsilon_i A_i^2} \bigr) \frac{\omega^3}{3!} \\
		& =  -\frac{4}{A_1^2A_2^2A_3^2} \bigl( \sumcyclic \epsilon_i p_i A_j^2 A_k^2 \bigr) \frac{\omega^3}{3!} .
	\end{align*}
	The result follows from using the observations in Remark \ref{rem:Chern_Clases}. We now turn to the computation of the slopes. Using the fact that $c_1(V_\rho)= c_1(L_{\rho_1}) + c_1(L_{\rho_2})$, we find that 
	\begin{align*}
		\mu(L_{\rho_1})- \mu(V_\rho) & = \int_{\bbF_2} c_1(L_{\rho_1}) \wedge \omega^2 - \frac{1}{2}\int_{\bbF_2} c_1(V_\rho) \wedge \omega^2 \\
		& = \frac{1}{2} \int_{\bbF_2} \left( c_1(L_{\rho_1}) - c_1(L_{\rho_2}) \right) \wedge \omega^2.
	\end{align*}
	Inserting the previous expression for $\left( c_1(L_{\rho_1}) - c_1(L_{\rho_2}) \right) \wedge \omega^2$ yields the stated result. The case of $\mu(L_{\rho_2})- \mu(V_\rho)$ follows from a similar computation.
\end{proof}

\begin{cor}\label{cor:Cones}
	Let $\rho_1-\rho_2 = \beta_m$ with $\sign(m_1)=-\sign(m_3)$ and consider $V_\rho=L_{\rho_1} \oplus L_{\rho_2}$. Then, there is a non-empty open cone $\mathcal{K}_1$ of K\"ahler classes $[\omega]$ for which $\mu(L_{\rho_1})- \mu(V_\rho)>0> \mu(L_{\rho_2})- \mu(V_\rho)$ and a non-empty open cone $\mathcal{K}_2$ of K\"ahler classes $[\omega]$ for which $\mu(L_{\rho_2})- \mu(V_\rho)>0> \mu(L_{\rho_1})- \mu(V_\rho)$.
\end{cor}
\begin{proof}
	The sign of $\mu(L_{\rho_1})- \mu(V_\rho)$ coincides with that of the quantity
	\begin{align*}
		\sumcyclic \epsilon_i m_i A_j^2 A_k^2 & = m_1 A_2^2 A_3^2 - m_2 A_3^2 A_1^2  +  m_3 A_1^2 A_2^2  \\
		& = m_1 A_2^2 A_3^2 + (m_1+m_3) A_3^2 A_1^2  +  m_3 A_1^2 A_2^2 \\
		& = m_1 ( A_2^2 A_3^2 + A_3^2 A_1^2 )  +  m_3 ( A_1^2 A_2^2 + A_3^2 A_1^2) \\
		& = m_1 A_3^2 (A_2^2+A_1^2 ) + m_3 A_1^2 (A_2^2+A_3^2) \\
		& = m_1 A_3^2 (A_3^2+2A_1^2 ) + m_3 A_1^2 (A_1^2+2A_3^2) \\
		& = m_1 (A_3^4 + 2A_1^2A_3^2) + m_3 (A_1^4+2 A_1^2A_3^2) .
	\end{align*}
	The result follows from noticing that if for example $m_1$ is negative, then for $\frac{A_3^4}{A_1^4}$ sufficiently large this quantity is negative. 
\end{proof}

\begin{remark}
The condition of Corollary \ref{cor:Cones} is satisfied for instance if $i\beta_m = \pm r_1$ or $i\beta_m = \pm r_3$.	
\end{remark}

To continue with the computation of the curvature we must  explicitly write $\psi_{\pm i}$. In this way, we can write
\begin{align*}
		A & = \frac{\beta_p}{2} \otimes \TT_0 +  \frac{\beta_m}{2} \otimes \TT_1 \\
		&\quad + \sum_{i=1}^3  a_{+i} \left[  (\cos(b_{+i})\theta_i + \sin(b_{+i}) \eta_i )  \otimes \TT_2 +  (\cos(b_{+i})\eta_i - \sin(b_{+i}) \theta_i ) \otimes \TT_{3} \right]  \\
		&\quad + \sum_{i=1}^3 a_{-i} \left[  (\cos(b_{-i}) \eta_i + \sin(b_{-i}) \theta_i) \otimes \TT_2 + ( \cos(b_{-i}) \theta_i - \sin(b_{-i}) \eta_i ) \otimes \TT_{3} \right.
\end{align*}
By further making use of an invariant gauge transformation, we can reduce to the case in which $b_{\pm i}=0$ (recall that at most one of the $a_{\pm i}$ can be nonzero). Hence, the connection reads 
\begin{align*}
	A & = \frac{\beta_p}{2} \otimes \TT_0 +  \frac{\beta_m}{2} \otimes \TT_1  + \sum_{i=1}^3  a_{+i} \left[ \theta_i  \otimes \TT_2 +  \eta_i\otimes \TT_{3} \right]   + \sum_{i=1}^3 a_{-i} \left[  \eta_i \otimes \TT_2 + \theta_i \otimes \TT_{3} \right] ,
\end{align*}
and so 
\[\mA_0=\frac{\beta_p}{2}, \ \ \mA_1 = \frac{\beta_m}{2} , \ \ \mA_2= \sum_{i=1}^3 ( a_{+i} \theta_i + a_{-i} \eta_i ) , \ \ \mA_3 = \sum_{i=1}^3 (a_{+i} \eta_i + a_{-i} \theta_i )  . \]
For simplicity, we shall assume only the $a_{+i}=a_i$ to be nonzero. However, the remaining case follows from a similar computation.
Using the equation we found for the components of the curvature gives
\begin{align*}
	F_0 & =  - \frac{\ii}{2}\sum_{i=1}^3 \frac{p_i}{\epsilon_i A_i^2} \left(  \alpha_i \wedge \overline{\alpha}_i \right),\\
	F_1 & =  -  \frac{\ii}{2}\sum_{i=1}^3 \frac{m_i + 2a_i^2}{\epsilon_i A_i^2} \left( \alpha_i \wedge \overline{\alpha}_i \right),   \\ 
	F_2 & =  \frac{a_1}{\epsilon_2 \epsilon_3 A_2 A_3} \Re( \overline{\alpha}_2 \wedge \alpha_3 ) - \frac{a_2}{\epsilon_1 \epsilon_3 A_1 A_3} \Re( \alpha_3 \wedge \alpha_1 ) + \frac{a_3}{\epsilon_1 \epsilon_2 A_1 A_2} \Re( \overline{\alpha}_1 \wedge \alpha_2 ), \\ 
	F_3 & =  \frac{a_1}{\epsilon_2 \epsilon_3 A_2 A_3} \Im( \overline{\alpha}_2 \wedge \alpha_3 ) - \frac{a_2}{\epsilon_1 \epsilon_3 A_1 A_3} \Im( \alpha_3 \wedge \alpha_1 ) - \frac{a_3}{\epsilon_1 \epsilon_2 A_1 A_2} \Im( \overline{\alpha}_1 \wedge \alpha_2 ).
\end{align*}
In particular, from these formula we see that $F_A$ is of type $(1,1)$ if and only if $a_2=0$ leading to the following conclusion.

\begin{lem}\label{lem:Hol bundle}
	The connection $A$ induces a holomorphic structure on $V_\rho$  if and only if $a_2=0$.
\end{lem}

As a consequence, we only need to consider the situation in which $\ii(\rho_1-\rho_2) \neq \pm r_2$, and thus  $a_2= 0$ by Schur's lemma.

\begin{remark}\label{rem:Weyl} 
Furthermore, we notice that there is a an element of the Weyl group (reflection on $r_2$) which preserves the complex structure $J$ on $\bbF_2$ and interchanges $r_1$ and $r_3$. Hence, up to the action of the Weyl group there is no loss of generality in also assuming that $a_3 = 0$.

Furthermore, the Weyl group also interchanges $r_1$ with $-r_1$ and so, it is enough to consider the case in which $\ii(\rho_1-\rho_2)= + r_1$.
\end{remark}

In view of Lemma \ref{lem:Hol bundle} and Remark \ref{rem:Weyl}, we shall now restrict ourselves to analysing the case
\[\ii(\rho_1-\rho_2)=  r_1=\ii\beta_{(-2,1)},\]
and we shall write $a_1=a$. Then, it follows from our previous discussion that, up to an invariant gauge transformation, the most general invariant connection can be written as
\begin{equation}\label{eq:invariant connection}
	A  = \frac{\rho_1 + \rho_2}{2} \otimes \TT_0 +  \frac{\rho_1 - \rho_2}{2} \otimes \TT_1  + a \left(  \theta_1 \otimes \TT_2 + \eta_1  \otimes \TT_{3} \right) . 
\end{equation}
Using the notation $p_1+p_2+p_3=0$, such a connection can be written as
\[\mA_0 = \frac{\beta_p}{2}, \ \ \mA_1= \frac{\beta_m}{2} = - \frac{\ii r_1}{2} , \ \ \mA_2 = a \theta_1 , \ \ \mA_3 = a\eta_1 .\]
Inserting this into the equations for the curvature previously found, we obtain the following.

\begin{lem}\label{lem:Curvature r_1}
	Let $V_\rho$ be such that $\ii(\rho_1-\rho_2)= r_1$, $\rho_1+\rho_2=\beta_p$ and $A$ be the invariant connection \eqref{eq:invariant connection}. Then, its curvature can be written as $F_A= \sum_{\mu=0}^3 F_\mu \otimes T_\mu$ with
	\begin{align*}
		F_0 & =  - \frac{\ii}{2}\left(  \frac{p_1}{A_1^2}\alpha_1 \wedge \overline{\alpha}_1  -  \frac{p_2}{A_2^2}  \alpha_2 \wedge \overline{\alpha}_2  +  \frac{p_3}{A_3^2}\alpha_3 \wedge \overline{\alpha}_3\right) ,\\
	F_1 & =  - \frac{2\ii(a^2-1)}{2A_1^2} \alpha_1 \wedge \overline{\alpha}_1 + \frac{\ii}{2A_2^2} \alpha_2 \wedge \overline{\alpha}_2 - \frac{\ii}{2A_3^2} \alpha_3 \wedge \overline{\alpha}_3,    \\ 
	F_2 & = \frac{a}{\epsilon_2 \epsilon_3 A_2 A_3} \Re( \overline{\alpha}_2 \wedge \alpha_3 ), \\ 
	F_3 & =  \frac{a}{\epsilon_2 \epsilon_3 A_2 A_3} \Im ( \overline{\alpha}_2 \wedge \alpha_3 ) .
	\end{align*}
\end{lem}

\subsection{The dHYM equation for invariant connections}

Now we turn to the dHYM equation, which is given by
\begin{equation}
	\Im \left(e^{-\ii\theta} (\omega \otimes \id_{\mathrm{U}(2)} + F_A)^3 \right) =  0.
\end{equation}
To work out this equation it is convenient to define
\begin{equation}
	\mE:=  (\omega \otimes \id_{\mathrm{U}(2)} + F_A)^3 ,
\end{equation}
which one can write as
$$\mE=E_0 \otimes \id_{\U(2)} + E_1 \otimes \TT_1 + E_2 \otimes \TT_2 + E_3 \otimes \TT_3,$$
for some $E_\mu \in \Omega^1(\SU(3), \bbC)$.

\begin{lem}\label{lem:E}
	One can write the quantities $E_\mu \in \Omega^1(\SU(3), \bbC)$ above as
	\begin{align*}
		E_0 & = \left(\omega^3 - 3 \sum_{\mu=0}^3 \omega \wedge F_\mu^2 \right)  + \ii \left( 3\omega^2 \wedge F_0 - F_0^3 - 3 \sum_{i=1}^3 F_i^2 \wedge  F_0  \right)  \\
		E_i & =  3 \omega^2 \wedge F_i -  3 F_0^2 \wedge F_i -  F_i \wedge  \sum_{j=1}^3 F_j^2 + \ii \ 6 \omega \wedge F_0 \wedge F_i   \\
		& = F_i \wedge \left( 3 \omega^2-  3 F_0^2  -  \sum_{j=1}^3 F_j^2 + \ii \ 6 \omega \wedge F_0   \right) ,
	\end{align*}
	for $i \in \{1,2,3\}$. In particular, if $E_2=0=E_3$, the connection $A$ is dHYM if and only if 
	\begin{equation}\label{eq:argument}
		\Im (E_0) \Im(E_1) = - \Re(E_0)\Re(E_1).
	\end{equation}
\end{lem}
\begin{proof}
Expanding the cube gives
\begin{equation}
	\mE = \omega^3 \otimes \id_{\mathrm{U}(2)}  + 3 \omega^2 \wedge F_A +  3\omega \wedge F_A^2 + F_A^3 ,
\end{equation}
and we shall now compute each of these terms. The second of these is simply given by
\begin{equation*}
		3\omega^2 \wedge F_A  = 3 \sum_{\mu=0}^3 (\omega^2 \wedge F_\mu) \otimes \TT_\mu .
\end{equation*}

On the other hand, using the facts that $\TT_0=i \ \id_{\U(2)}$, $\TT_i \TT_j = - \TT_j \TT_i$ for $i \neq j$, and $\TT_\mu^2 = - \id_{\U(2)}$, $\mu=0,1,2,3$, we obtain
\begin{align*}
	F_A^2 & = \biggl(F_0 \otimes \TT_0 + \sum_{i=1}^3 F_i \otimes \TT_i \biggr)^2 \\
	& = - F_0^2 \otimes \id_{\U(2)}   + 2 \ii \sum_{i=1}^3 (  F_0 \wedge F_i ) \otimes \TT_i - \sum_{i=1}^3 F_i^2 \otimes \id_{\U(2)} \\
	& = - \sum_{\mu=0}^3 F_\mu^2 \otimes \id_{\U(2)} + \ii \sum_{i=1}^3 2  F_0 \wedge F_i  \otimes \TT_i,
\end{align*}
and therefore
\begin{align*}
	F_A^3 & = \biggl( - \sum_{\mu=0}^3 F_\mu^2 \otimes \id_{\U(2)} + \ii \sum_{i=1}^3 2  F_0 \wedge F_i  \otimes \TT_i \biggr) \wedge \biggl( F_0 \otimes \TT_0 + \sum_{j=1}^3 F_j \otimes \TT_j \biggr) \\
	& = - \ii  F_0 \wedge \sum_{\mu=0}^3 F_\mu^2 \otimes \id_{\U(2)} - \sum_{i=1}^3 F_i \wedge  \sum_{\mu=0}^3 F_\mu^2  \otimes \TT_i -  \sum_{i=1}^3 2 F_0^2 \wedge F_i  \otimes \TT_i \\
	& \ \ \ - \ii \sum_{i=1}^3 2 F_0 \wedge F_i^2 \otimes \id_{\U(2)} \\
	& =  - \ii \left[ F_0 \wedge \sum_{\mu=0}^3 F_\mu^2 + \sum_{i=1}^2 2 F_i^2 \wedge  F_0  \right]  \otimes \id_{\U(2)}  - \sum_{i=1}^3 \left[ F_i \wedge  \sum_{\mu=0}^3 F_\mu^2  + 2 F_0^2 \wedge F_i  \right]\otimes \TT_i ,
\end{align*}
which we can also write as
\begin{equation*}
	F_A^3 = - \ii \left[ F_0^3 + 3 \sum_{i=1}^2 F_i^2 \wedge  F_0  \right]  \otimes \id_{\U(2)}  - \sum_{i=1}^3 \left[ 3 F_0^2 \wedge F_i +  F_i \wedge  \sum_{j=1}^3 F_j^2  \right]\otimes \TT_i. 
\end{equation*}
Furthermore, from the previous computation of $F_A^2$ we also find that
\begin{equation*}
	3\omega \wedge F_A^2  = - 3 \sum_{\mu=0}^3 (\omega \wedge F_\mu^2) \otimes \id_{\U(2)} + 6\ii \sum_{i=1}^3 (\omega \wedge F_0 \wedge F_i) \otimes \TT_i  
\end{equation*}
Putting all these formulas together we obtain
\begin{align*}
	E_0 & = \left(\omega^3 - 3 \sum_{\mu=0}^3 \omega \wedge F_\mu^2 \right)  + \ii \left( 3\omega^2 \wedge F_0 - F_0^3 - 3 \sum_{i=1}^2 F_i^2 \wedge  F_0  \right)  \\
	E_i & =  3 \omega^2 \wedge F_i -  3 F_0^2 \wedge F_i -  F_i \wedge  \sum_{j=1}^3 F_j^2 + 6\ii \omega \wedge F_0 \wedge F_i   \\
	& = F_i \wedge \left( 3 \omega^2-  3 F_0^2  -  \sum_{j=1}^3 F_j^2 +  6\ii \omega \wedge F_0   \right) ,
\end{align*}
which are the formulas in the statement.

Now, we turn to the case in which $E_2=0=E_3$. In that case, we have
\[\mE= E_0 \otimes \id_{\U(2)} + E_1 \otimes \TT_1 = \begin{pmatrix}
	E_0+\ii E_1 & 0 \\
	0 & E_0 - \ii E_1
\end{pmatrix}.\]
Thus, in this case, $A$ is dHYM if and only if $\Im\bigl(e^{-\ii\theta}(E_0\pm \ii E_1)\bigr)=0$, and thus if and only if
$\Im(e^{-\ii\theta}E_0)=\Im(e^{-\ii\theta}\ii E_1)=0$, and thus if and only if
\[\arg (E_0)= \arg (\ii E_1) \mod \pi\]
which can equally be written as	$\Im (E_0) \Im(E_1) = - \Re(E_0)\Re(E_1)$ as stated.
\end{proof}

\subsection{Solving the equations}
We continue to operate under the asumption, as per Remark \ref{rem:Weyl}, that $\ii(\rho_1-\rho_2)=r_1=\beta_{(2,-1)}$. We then know that $m=(m_1,m_3)=(-2,1)$ and thus $m_2=1$, and that $\rho_1+\rho_2=\beta_p$. We can rewrite the result of Lemma \ref{lem:Curvature r_1} succintly by introducing additional notation:
	\begin{align*}
		\omega_j & = \frac{\ii}{2} \alpha_j \wedge \overline{\alpha}_j, \ \text{for $j=1,2,3$}, \\
		R_{23} & = \frac{1}{\epsilon_2 \epsilon_3 A_2 A_3} \Re( \overline{\alpha}_2 \wedge \alpha_3 ), \\
		I_{23} & = \frac{1}{\epsilon_2 \epsilon_3 A_2 A_3} \Im ( \overline{\alpha}_2 \wedge \alpha_3 ).
	\end{align*}
We then have
\begin{align*}
	F_0 & =  \frac{d\beta_p}{2}, 
	&F_1 & =   \frac{d\beta_m }{2} - \frac{2a^2}{\epsilon_1 A_1^2} \omega_1,  \\ 
	F_2 & =  a R_{23}, 
	&F_3 & =  a I_{23}.
\end{align*}

	Using this notation we have
	\[\frac{d\beta_p}{2}  = -  \frac{p_1}{\epsilon_1 A_1^2} \omega_1  -  \frac{p_2}{\epsilon_2 A_2^2} \omega_2  -  \frac{p_3}{\epsilon_3 A_3^2} \omega_3,\]
	and $\omega^3 = 3! \ \omega_1 \wedge \omega_2 \wedge \omega_3$. We shall now state some useful equations regarding $R_{23}$, and $I_{23}$. These satisfy 
	\[R_{23} \wedge I_{23}=0,\]
	and
\begin{align*}
		R_{23} \wedge \omega^2 = & 0  = I_{23} \wedge \omega^2, \\
		R_{23} \wedge d \beta_p \wedge \omega = & 0  = I_{23} \wedge d \beta_p \wedge \omega, \\
		R_{23} \wedge (d \beta_p)^2 = & 0  = I_{23} \wedge (d \beta_p)^2 .
\end{align*}
Also, from the fact that
$F_1^2 = \left(\frac{d\beta_m}{2}\right)^2 - 4 a^2 \frac{\omega_1}{A_1^2} \wedge \frac{d \beta_m}{2},$
we further compute that
\[R_{23} \wedge F_1^2=0=I_{23} \wedge F_1^2.\]
Using these identities and Lemma \ref{lem:E}, it is a straightforward computation to check that $E_2=0=E_3$. We then know from Lemma \ref{lem:E} that $A$ is dHYM if and only if Equation \eqref{eq:argument} is satisfied. We thus need to compute $E_0$ and $E_1$. In order to compute them we derive the following identities.
\begin{lem}\label{lem:componentsE0E1}
	We continue under the assumption that $\ii(\rho_1-\rho_2)=r_1$, as previously stated.
	The terms which appear in $\Re(E_0)$ are given by
	\begin{align*}
		\omega \wedge F_0^2 & = 2 \frac{\sumcyclic \epsilon_i A_i^2 p_jp_k}{\epsilon_1 A_1^2 \epsilon_2 A_2^2 \epsilon_3 A_3^2} \ \frac{\omega^3}{3!}, \\
		\omega \wedge F_1^2 & = \frac{2}{\epsilon_1 A_1^2 \epsilon_2 A_2^2 \epsilon_3 A_3^2} \biggl( \sumcyclic \epsilon_i A_i^2 m_j m_k + 2a^2 ( m_2\epsilon_3 A_3^2 + m_3\epsilon_2 A_2^2 ) \biggr)  \frac{\omega^3}{3!}, \\
		\omega \wedge F_2^2 = \omega \wedge F_3^2 & = 2 \frac{\epsilon_1 A_1^2 a^2}{\epsilon_1 A_1^2 \epsilon_2 A_2^2 \epsilon_3 A_3^2} \frac{\omega^3}{3!} ,
	\end{align*}
	while those appearing in $\Im(E_0)$ are
	\begin{align*}
		\omega^2 \wedge F_0 & = - \frac{2}{\epsilon_1 A_1^2 \epsilon_2 A_2^2 \epsilon_3 A_3^2} \biggl( \sumcyclic p_i \epsilon_j A_j^2 \epsilon_k A_k^2 \biggr) \frac{\omega^3}{3!}, \\
		F_0^3 & = - 6 \prod_{i=1}^3\frac{p_i}{\epsilon_i A_i^2} \ \frac{\omega^3}{3!}, \\
		F_1^2 \wedge F_0 & = - 2\frac{\sumcyclic p_i m_j m_k + 2a^2  (p_2m_3+p_3m_2) }{\epsilon_1 A_2^2 \epsilon_2 A_2^2 \epsilon_3 A_3^2} \  \frac{\omega^3}{3!}, \\
		F_2^2 \wedge F_0 = F_3^2 \wedge F_0 & = -a^2 \frac{2p_1}{\epsilon_1 A_1^2 \epsilon_2 A_2^2 \epsilon_3 A_3^2} \frac{\omega^3}{3!} .
	\end{align*}
	
	On the other hand, the terms appearing in $\Re(E_1)$ are
	\begin{align*}
		\omega^2 \wedge F_1 & = - \frac{2}{\epsilon_1 A_1^2 \epsilon_2 A_2^2 \epsilon_3 A_3^2} \biggl( \sumcyclic m_i \epsilon_j A_j^2 \epsilon_k A_k^2 \ + 2a^2 \epsilon_2 A_2^2 \epsilon_3 A_3^2  \biggr) \frac{\omega^3}{3!}, \\
		F_0^2 \wedge F_1 & = - 2\frac{\sumcyclic m_i p_j p_k + 2a^2  p_2p_3 }{\epsilon_1 A_2^2 \epsilon_2 A_2^2 \epsilon_3 A_3^2} \  \frac{\omega^3}{3!}, \\
		F_1^3 & = - \frac{6m_2m_3}{\epsilon_1 A_2^2 \epsilon_2 A_2^2 \epsilon_3 A_3^2} \left( m_1 + 2a^2 \right) \frac{\omega^3}{3!}, \\
		F_2^2 \wedge F_1 = F_3^2 \wedge F_1 & = \frac{ \epsilon_2 \epsilon_3}{\epsilon_1 A_1^2 \epsilon_2 A_2^2 \epsilon_3 A_3^2} 2a^2 (2a^2+m_1)  \frac{\omega^3}{3!} ,
	\end{align*}
	while the term appearing in $\Im(E_1)$ is given by
	\begin{align*}
		\omega \wedge F_0 \wedge F_1 & =  \frac{1}{\epsilon_1 A_1^2 \epsilon_2 A_2^2 \epsilon_3 A_3^2} \biggl( \sumcyclic m_k ( p_i \epsilon_j A_j^2 + p_j \epsilon_i A_i^2  ) + 2a^2 ( p_2\epsilon_3 A_3^2 + p_3\epsilon_2 A_2^2  ) \biggr) \frac{\omega^3}{3!} .
	\end{align*}
\end{lem}
\begin{proof}
We start by noticing that
\[F_0^2 = \left( \frac{d\beta_p}{2} \right)^2 =\sumcyclic \frac{2p_j p_k}{\epsilon_j A_j^2 \epsilon_k A_k^2} \ \omega_j \wedge \omega_k , \]
This equation can also be used to compute
\[F_0^3 = \left( \frac{d\beta_p}{2} \right)^3 = - 6 \prod_{i=1}^3\frac{p_i}{\epsilon_i A_i^2} \ \omega_1 \wedge \omega_2 \wedge \omega_3 = - \prod_{i=1}^3\frac{p_i}{\epsilon_i A_i^2} \ \omega^3,\]
and
\begin{align*}
	\left( \frac{d\beta_p}{2} \right)^2 \wedge  \frac{d\beta_m}{2}  & = - \sumcyclic \frac{2m_i p_j p_k}{\epsilon_i A_i^2 \epsilon_j A_j^2 \epsilon_k A_k^2} \ \omega_1 \wedge \omega_2 \wedge \omega_3 =  - \sumcyclic \frac{2m_i p_j p_k}{\epsilon_i A_i^2 \epsilon_j A_j^2 \epsilon_k A_k^2} \ \frac{\omega^3}{3!}, \\
	\left( \frac{d\beta_p}{2} \right)^2 \wedge  \frac{2\omega_1}{A_1^2} & = \frac{2}{A_1^2} \frac{2p_2p_3}{\epsilon_2 A_2^2 \epsilon_3 A_3^2}  \ \frac{\omega^3}{3!}.
\end{align*}
These equations further imply that
\begin{align*}
	F_0^2 \wedge F_1 & = \left( \frac{d\beta_p}{2} \right)^2 \wedge  \frac{d\beta_m}{2} -  a^2 \left( \frac{d\beta_p}{2} \right)^2 \wedge  \frac{2\omega_1}{A_1^2} \\
	& = \biggl( -\sumcyclic \frac{2m_i p_j p_k}{\epsilon_i A_i^2 \epsilon_j A_j^2 \epsilon_k A_k^2} - \frac{2a^2}{A_1^2}  \frac{2p_2p_3}{\epsilon_2 A_2^2 \epsilon_3 A_3^2} \biggr)  \frac{\omega^3}{3!} \\
	& = - 2\frac{\sumcyclic m_i p_j p_k + 2a^2  p_2p_3 }{\epsilon_1 A_2^2 \epsilon_2 A_2^2 \epsilon_3 A_3^2} \  \frac{\omega^3}{3!},
\end{align*}
and 
\begin{align*}
	F_1^2 \wedge F_0 & = \left( \frac{d\beta_m}{2} \right)^2 \wedge  \frac{d\beta_p}{2} -   2a^2 \frac{2\omega_1}{A_1^2} \wedge \frac{d \beta_m}{2} \wedge \frac{d\beta_p}{2} \\
	& = \biggl( - \sumcyclic \frac{2p_i m_j m_k}{\epsilon_i A_i^2 \epsilon_j A_j^2 \epsilon_k A_k^2} - 2a^2 \frac{2}{A_1^2} \frac{p_2m_3 +p_3m_2}{\epsilon_2 A_2^2 \epsilon_3 A_3^2} \biggr)  \frac{\omega^3}{3!}\\
	& = - 2\frac{\sumcyclic p_i m_j m_k + 2a^2  (p_2m_3+p_3m_2) }{\epsilon_1 A_2^2 \epsilon_2 A_2^2 \epsilon_3 A_3^2} \  \frac{\omega^3}{3!}
\end{align*}
and
\begin{align*}
	F_1^3 & = \biggl( \bigl(\frac{d\beta_m}{2}\bigr)^2 -2 a^2 \frac{2\omega_1}{A_1^2} \wedge \frac{d \beta_m}{2} \biggr) \wedge \biggl( \frac{d\beta_m }{2} - \frac{2a^2}{A_1^2} \omega_1 \biggr) \\
	& = \bigl(\frac{d\beta_m}{2}\bigr)^3 - 3 \times \frac{2a^2}{A_1^2} \omega_1 \wedge \bigl(\frac{d\beta_m}{2}\bigr)^2  \\
	& = - \prod_{i=1}^3\frac{m_i}{\epsilon_i A_i^2} \ \omega^3 - 3 \times \frac{2a^2}{A_1^2} \frac{2m_2m_3}{\epsilon_2 A_2^2 \epsilon_3 A_3^2} \omega_1 \wedge \omega_2 \wedge \omega_3 \\
	& = - \bigl( 6\prod_{i=1}^3\frac{m_i}{\epsilon_i A_i^2} + a^2  \frac{12}{A_1^2} \frac{m_2m_3}{\epsilon_2 A_2^2 \epsilon_3 A_3^2} \bigr) \frac{\omega^3}{3!} \\
	& = - \frac{6}{\epsilon_1 A_2^2 \epsilon_2 A_2^2 \epsilon_3 A_3^2} \left( m_1m_2m_3 + 2 a^2 m_2m_3 \right) \frac{\omega^3}{3!} \\
	& = - \frac{6m_2m_3}{\epsilon_1 A_2^2 \epsilon_2 A_2^2 \epsilon_3 A_3^2} \left( m_1 + 2 a^2 \right) \frac{\omega^3}{3!}.
\end{align*}
Next, we further compute
\begin{align*}
	\omega \wedge F_0^2 & = \sumcyclic \frac{2p_j p_k}{\epsilon_j A_j^2 \epsilon_k A_k^2} \ \frac{\omega^3}{3!} \\
	& = 2 \frac{\sumcyclic \epsilon_i A_i^2 p_jp_k}{\epsilon_1 A_2^2 \epsilon_2 A_2^2 \epsilon_3 A_3^2} \ \frac{\omega^3}{3!} ,
\end{align*}
and	
\begin{align*}
	\omega \wedge F_1^2 & = \sumcyclic \frac{2m_j m_k}{\epsilon_j A_j^2 \epsilon_k A_k^2} \ \frac{\omega^3}{3!} - 2a^2 \frac{2\omega_1}{A_1^2} \wedge \frac{d\beta_m}{2} \wedge \omega \\
	& = \biggl( \sumcyclic \frac{2m_j m_k}{\epsilon_j A_j^2 \epsilon_k A_k^2} + 2a^2 \frac{2}{A_1^2} \bigl( \frac{m_2}{\epsilon_2 A_2^2} + \frac{m_3}{\epsilon_3 A_3^2} \bigr) \biggr)  \frac{\omega^3}{3!} \\
	& = \frac{2}{\epsilon_1 A_2^2 \epsilon_2 A_2^2 \epsilon_3 A_3^2} \biggl( \sumcyclic \epsilon_i A_i^2 m_j m_k + 2a^2 ( m_2\epsilon_3 A_3^2 + m_3\epsilon_2 A_2^2 ) \biggr)  \frac{\omega^3}{3!} ,
\end{align*}
and, since $\epsilon_1=1$,
\begin{align*}
	\omega^2 \wedge F_1 & = 2 (\omega_2\wedge\omega_3 +\omega_1\wedge\omega_1  + \omega_1\wedge\omega_1) \wedge \left( \frac{d\beta_m}{2} - \frac{2a^2}{A_1^2} \omega_1 \right) \\
	& = -2 \left( \sum_{i=1}^3 \frac{m_i}{\epsilon_i A_i^2} + \frac{2a^2}{\epsilon_1A_1^2} \right) \frac{\omega^3}{3!}\\
	& = - \frac{2}{\epsilon_1 A_1^2 \epsilon_2 A_2^2 \epsilon_3 A_3^2} \left( \sumcyclic m_i \epsilon_j A_j^2 \epsilon_k A_k^2  + 2a^2 \epsilon_2 A_2^2 \epsilon_3 A_3^2  \right) \frac{\omega^3}{3!}.
\end{align*}
On the other hand,  using that $\epsilon_j \epsilon_k=-\epsilon_i$ for any cyclic permutation $(i,j,k)$ of $(1,2,3)$, we find that
\begin{align*}
	\omega^2 \wedge F_1 &  = - \frac{2}{\epsilon_1 A_1^2 \epsilon_2 A_2^2 \epsilon_3 A_3^2} \biggl( - \sumcyclic m_i \epsilon_i A_j^2 A_k^2  + 2a^2 \epsilon_2 A_2^2 \epsilon_3 A_3^2  \biggr) \frac{\omega^3}{3!},
\end{align*}
Furthermore, using $\epsilon_1=1$ and $\epsilon_j \epsilon_k=-\epsilon_i$ again, we also compute
\begin{align*}
	\omega^2 \wedge F_0 & = 2 (\omega_{23} + \omega_{31} + \omega_{12}) \wedge \bigl( \frac{d\beta_p}{2} \bigr) \\
	& = -2 \biggl( \sum_{i=1}^3 \frac{p_i}{\epsilon_i A_i^2} \biggr) \frac{\omega^3}{3!} \\
	& = - \frac{2}{\epsilon_1 A_1^2 \epsilon_2 A_2^2 \epsilon_3 A_3^2} \biggl( \sumcyclic p_i \epsilon_j A_j^2 \epsilon_k A_k^2 \biggr) \frac{\omega^3}{3!} \\
	& =  \frac{2}{\epsilon_1 A_1^2 \epsilon_2 A_2^2 \epsilon_3 A_3^2} \biggl( \sumcyclic p_i \epsilon_i A_j^2 A_k^2 \biggr) \frac{\omega^3}{3!} .
\end{align*}
We now turn to the computation of $\omega \wedge F_0 \wedge F_1$ using $\epsilon_1=1$ once again
\begin{align*}
	\omega \wedge F_0 \wedge F_1 & = \omega \wedge \frac{d\beta_p}{2} \wedge \left(\frac{d\beta_m}{2} - \frac{2a^2}{A_1^2} \omega_1 \right) \\
	& = - \sumcyclic \frac{p_i}{\epsilon_i A_i^2} \omega_i \wedge (\omega_j + \omega_k) \wedge \biggl(- \sum_{l=1}^3 \frac{m_l}{\epsilon_l A_l^2} \omega_l - \frac{2a^2}{A_1^2} \omega_1 \biggr) \\
	& = - \sum_{i <j} \bigl( \frac{p_i}{\epsilon_i A_i^2} + \frac{p_j}{\epsilon_j A_j^2}  \bigr) \omega_i \wedge \omega_j \wedge  \biggl(- \sum_{l=1}^3 \frac{m_l}{\epsilon_l A_l^2} \omega_l - \frac{2a^2}{\epsilon_1 A_1^2} \omega_1 \biggr) \\
	& = \biggl( \sumcyclic \frac{m_k}{\epsilon_k A_k^2} \bigl( \frac{p_i}{\epsilon_i A_i^2} + \frac{p_j}{\epsilon_j A_j^2}  \bigr) + \frac{2a^2}{\epsilon_1 A_1^2} \bigl( \frac{p_2}{\epsilon_2 A_2^2} + \frac{p_3}{\epsilon_3 A_3^2}  \bigr) \biggr) \frac{\omega^3}{3!}.
\end{align*}
This can be put in a more convenient form by reducing everything to the same denominator as follows
\begin{align*}
	\omega \wedge F_0 \wedge F_1 & = \frac{1}{\epsilon_1 A_1^2 \epsilon_2 A_2^2 \epsilon_3 A_3^2} \biggl( \sumcyclic m_k ( p_i \epsilon_j A_j^2 + p_j \epsilon_i A_i^2  ) + 2a^2 ( p_2\epsilon_3 A_3^2 + p_3\epsilon_2 A_2^2  ) \biggr) \frac{\omega^3}{3!}.
\end{align*}
Finally, we turn to the terms in the equations which involve $F_2$ and $F_3$ which recall are given by $aR_{23}$ and $aI_{23}$ respectively. Then, from $R_{23}^2 = - \frac{2}{A_2^2A_3^2} \omega_2 \wedge \omega_3 = I_{23}^2$, we find that
\[F_2^2 = - a^2 \frac{2}{A_2^2A_3^2} \omega_2 \wedge \omega_3 = F_3^2,\]
which in turn can be used to compute 
\[F_2^2 \wedge \omega = - a^2 \frac{2}{A_2^2A_3^2} \omega_1 \wedge \omega_2 \wedge \omega_3 = - a^2 \frac{2}{A_2^2A_3^2} \frac{\omega^3}{3!} = a^2 \frac{2}{\epsilon_2 A_2^2 \epsilon_3 A_3^2} \frac{\omega^3}{3!},\]
where we used the fact that $\epsilon_2=-1$ and $\epsilon_3=1$. The same equation holds for $F_2^2 \wedge \omega$. Next, we find that $F_2^2 \wedge F_0 = F_3^2 \wedge F_0$ which is given by
\begin{align*}
	F_2^2 \wedge F_0  = a^2 \frac{2p_1}{\epsilon_1 A_1^2 A_2^2A_3^2} \frac{\omega^3}{3!} = -a^2 \frac{2p_1}{\epsilon_1 A_1^2 \epsilon_2 A_2^2 \epsilon_3 A_3^2} \frac{\omega^3}{3!} ,
\end{align*}
where, again, $\epsilon_2=-1$ and $\epsilon_3=1$ were used. Further using the fact that $\epsilon_1=1$ we have
\begin{align*}
	F_2^2 \wedge F_1 & =  - a^2 \frac{2}{A_2^2A_3^2} \omega_2 \wedge \omega_3 \wedge \left( \frac{d\beta_m}{2} - \frac{2a^2}{A_1^2} \omega_1 \right) \\
	& = - a^2 \frac{2}{A_2^2A_3^2} \left( - \frac{m_1}{\epsilon_1 A_1^2} - \frac{2a^2}{\epsilon_1 A_1^2} \right) \omega_1 \wedge \omega_2 \wedge \omega_3 \\
	& =  a^2 \frac{2}{\epsilon_1 A_1^2 A_2^2A_3^2} \left( 2a^2 + m_1 \right) \frac{\omega}{3!} \\
	& =  \frac{\epsilon_2 \epsilon_3 2a^2 (2a^2+m_1)}{\epsilon_1 A_1^2 \epsilon_2 A_2^2 \epsilon_3 A_3^2}   \frac{\omega}{3!}
\end{align*}
and the same equation holds for $F_3^2 \wedge F_1$.
\end{proof}
\begin{cor}\label{cor:R0_I0_R1_I1}
	In the same notation as before:
	\begin{align*}
		R_0 & := \frac{\Re (E_0)}{\frac{\omega^3}{A_1^2  A_2^2  A_3^2}}  =  A_1^2 A_2^2 A_3^2 + \sumcyclic \epsilon_i A_i^2 (p_jp_k +m_jm_k) , \\
		I_0 & := \frac{\Im (E_0)}{\frac{\omega^3}{A_1^2 A_2^2 A_3^2}}  = - p_1p_2p_3 - \sumcyclic p_i ( m_jm_k + \epsilon_i A_j^2 A_k^2 ),
	\end{align*}
	and
	\begin{align*}
		R_1 & := \frac{\Re (E_1)}{\frac{\omega^3}{ A_1^2  A_2^2  A_3^2}}  = - m_1 m_2 m_3 - \sumcyclic m_i (p_jp_k + \epsilon_i A_j^2 A_k^2) \\
		& \phantom{:= \frac{\Re (E_1)}{\frac{\omega^3}{ A_1^2  A_2^2  A_3^2}}  =} - 2a^2 \left( m_2m_3 +p_2p_3 + A_2^2 A_3^2 + \frac{m_1}{3} \right) - \frac{(2a^2)^2}{3} , \\
		I_1 & : = \frac{\Im (E_1)}{\frac{\omega^3}{A_1^2 A_2^2 A_3^2}}  =  - \sumcyclic  \epsilon_i A_i^2 \left( p_j m_k + m_j p_k \right) - 2a^2 \left( p_2\epsilon_3 A_3^2 + p_3\epsilon_2 A_2^2  \right)  .
	\end{align*}
\end{cor}

\begin{proof}
	Direct computation using the results of the previous Lemma yield
	\begin{align*}
		\frac{\Re (E_0)}{\frac{\omega^3}{\epsilon_1 A_1^2 \epsilon_2 A_2^2 \epsilon_3 A_3^2}} & = \epsilon_1 A_1^2 \epsilon_2 A_2^2 \epsilon_3 A_3^2 - a^2 \left(m_2 \epsilon_3 A_3^2 + m_3 \epsilon_2 A_2^2 + \epsilon_1 A_1^2 \right)- \sumcyclic \epsilon_i A_i^2 (p_jp_k +m_jm_k) , \\
		\frac{\Im (E_0)}{\frac{\omega^3}{\epsilon_1 A_1^2 \epsilon_2 A_2^2 \epsilon_3 A_3^2}} & = - \sumcyclic p_i \epsilon_j A_j^2 \epsilon_k A_k^2 + p_1p_2p_3 + \sum_{i=1}^3 p_i m_jm_k + 2a^2 (p_1+p_2m_3+p_3m_2).
	\end{align*}
	We can now use the fact that $m_1=-2$, $m_2=1$, $m_3=1$, to compute that
	\[m_2 \epsilon_3 A_3^2 + m_3 \epsilon_2 A_2^2 + \epsilon_1 A_1^2 =  A_3^2 - A_2^2 + A_1^2=0,\]
	and
	\[p_1+p_2m_3+p_3m_2= p_1 +p_2+p_3=0.\]
	Substituting these, and using $\epsilon_i=-\epsilon_j \epsilon_k$ and $\epsilon_1 \epsilon_2 \epsilon_3 =-1$, yields the formulas in the statement for $\Re(E_0)$ and $\Im(E_0)$. The claimed formulas for $\Re(E_1)$ and $\Im(E_1)$ follow from a similarly computation. 
\end{proof}

We are now ready to tackle the dHYM equation. Let
	\begin{align}
		\AAA & := \frac{1}{3}  \left( A_1^2 A_2^2 A_3^2 + \sumcyclic \epsilon_i A_i^2 (p_jp_k +m_jm_k) \right) , \label{eq:defA}\\
		\BBB & := \left( A_1^2 A_2^2 A_3^2 + \sumcyclic \epsilon_i A_i^2 (p_jp_k +m_jm_k) \right) \left( m_2m_3 +p_2p_3 + A_2^2 A_3^2 + \frac{m_1}{3} \right) \notag \\*
		& \quad -  \left( p_2\epsilon_3 A_3^2 + p_3\epsilon_2 A_2^2  \right) \left( p_1p_2p_3 + \sumcyclic p_i ( m_jm_k + \epsilon_i A_j^2 A_k^2 ) \right) ,\label{eq:defB} \\
		\CCC & := \left( m_1 m_2 m_3 + \sumcyclic m_i (p_jp_k + \epsilon_i A_j^2 A_k^2) \right) \left( A_1^2 A_2^2 A_3^2 + \sumcyclic \epsilon_i A_i^2 (p_jp_k +m_jm_k) \right)\notag \\*
		& \quad - \left( p_1p_2p_3 + \sumcyclic p_i ( m_jm_k + \epsilon_i A_j^2 A_k^2 ) \right)  \sumcyclic  {\epsilon_i A_i^2} \left( p_j m_k + m_j p_k \right) .\label{eq:defC}
	\end{align}

\begin{proposition}
	Let $p=(p_1,p_3) \in \Zeven \times \Zodd$. Then, the connection $A$ from Equation \eqref{eq:invariant connection} on $V_\rho$ with $\rho_1-\rho_2=ir_1$ and $\rho_1+\rho_2=\beta_p$ is dHYM if and only if
	\begin{equation}\label{eq:Quadratic_Equation}
		(2a^2)^2 \AAA + 2a^2 \BBB + \CCC =0.
	\end{equation}
\end{proposition}
\begin{proof}
If $p=(p_1,p_3) \in \Zeven \times \Zodd$ and $\rho_1-\rho_2=ir_1$, $\rho_1+\rho_2=\beta_p$, then both $\rho_1$ and $\rho_2$ are integral and so we can consider the bundle $V_\rho$. Then, using the fact that $E_2=0=E_3$ as we have proven at the beginning of this subsection, the dHYM equation for $A$ becomes Equation \eqref{eq:argument}. Using the notation of Corollary \ref{cor:R0_I0_R1_I1}, we see that this equation is equivalent to
\begin{align*}
0&=R_0R_1+I_0I_1\\
&=
	\biggl( p_1p_2p_3 + \sumcyclic p_i ( m_jm_k + \epsilon_i A_j^2 A_k^2 ) \biggr) \biggl(  \sumcyclic  \epsilon_i A_i^2 ( p_j m_k + m_j p_k ) + 2a^2 ( p_2\epsilon_3 A_3^2 + p_3\epsilon_2 A_2^2  )  \biggr)  \\
& \quad  +\biggl( A_1^2 A_2^2 A_3^2 + \sumcyclic \epsilon_i A_i^2 (p_jp_k +m_jm_k) \biggr)  \\
&\quad\times \biggl( - m_1 m_2 m_3 - \sumcyclic m_i (p_jp_k + \epsilon_i A_j^2 A_k^2)  - 2a^2 \bigl( m_2m_3 +p_2p_3 + A_2^2 A_3^2 + \frac{m_1}{3} \bigr) - \frac{(2a^2)^2}{3} \biggr).
\end{align*}
One can then write this quadratic equation in $2a^2$ in terms of its coefficients $\AAA,\BBB,\CCC$ to complete the proof.
\end{proof}

\subsection{The large radius regime}
In this section, we prove existence of solutions in the large radius regime. We choose $B_i>0$ such that $B_2^2=B_1^2+B_3^2$ (to guarantee that the Hermitian structure is K\"ahler), and set $A_i=tB_i$, and examine the behaviour of the quadratic Equation \eqref{eq:Quadratic_Equation} when $t\ge 0$. The asymptotic behaviour of the coefficients $\AAA,\BBB,\CCC$ is examined in Lemma \ref{lem:Coefficient_Large_Radius}, and used in Theorem \ref{thm:Large_Radius} to prove the existence of solutions in this limit.

\begin{lem}\label{lem:Coefficient_Large_Radius}
	 Let $B_i>0$ be such that $B_2^2=B_1^2+B_3^2$ and consider $A_i=tB_i$ for $i \in \{ i,2,3\}$ with $t \gg 1$. Then, as $t\to\infty$,
	 \begin{align*}
	 	\AAA & = \frac{B_1^2 B_2^2 B_3^2}{3}t^6  + O(t^2) , \\
	 	\BBB & =   B_1^2 B_2^4 B_3^4 t^{10} +O(t^6) , \\
	 	\CCC & =  t^{10} B_1^2 B_2^2 B_3^2 \sum_{i=1}^3 \epsilon_i m_i B_j^2 B_k^2 + O(t^6) .
	 \end{align*}
	 In particular, $\AAA>0$ and $\BBB>0$ for large $t \gg 1$. Furthermore, we can also write $\CCC$ as
	 \begin{align*}
	 	\CCC & =  t^{10} B_1^2 B_2^2 B_3^2 \bigl( m_1 B_3^2 (B_3^2+2B_1^2 ) + m_3 B_1^2 (B_1^2+2B_3^2) \bigr) + O(t^6) ,
	 \end{align*}
	 and 
	 \begin{equation}\label{eq:signCvsslopecomparison}
		 	 	\sign(\CCC) = \sign \bigl( \mu(L_{\rho_1})- \mu(V_\rho) \bigr)\quad \text{ for large }t \gg 1
	 \end{equation}
\end{lem}
\begin{proof}
	The proof follows from simply replacing $A_i=tB_i$, for $i \in \{ i,2,3\}$, into the formulas for $\AAA$, $\BBB$, and $\CCC$ given by Equations \eqref{eq:defA}, \eqref{eq:defB}, and \eqref{eq:defC}. 
	Finally, we notice that for such large $t \gg 1$, the sign of $\CCC$ coincides with that of the quantity $\sumcyclic \epsilon_i m_i B_j^2 B_k^2 $ which in turn coincides with that of $\mu(L_{\rho_1})- \mu(V_\rho)$ as computed in Lemma \ref{lem:Degrees}. 
\end{proof}

Recall that $\mathcal{K}_2$ represents the non-empty cone defined in Corollary \ref{cor:Cones} by the condition that $\mu(L_{\rho_1}) < \mu(V_\rho)< \mu(L_{\rho_1})$

\begin{thm}\label{thm:Large_Radius}
	Consider the complex vector bundle $V_\rho=L_{\rho_1} \oplus L_{\rho_2}$ with $i(\rho_1-\rho_2)=r_1$ and any K\"ahler class $[\omega] \in \mathcal{K}_2$. 
	Then, there is $T>0$ such that for all $t>T$, there is an irreducible deformed Hermitian--Yang--Mills connection with respect to the K\"ahler form $t\omega$ on $V_\rho$. 
\end{thm}

\begin{proof}
	For an irreducible dHYM connection to exist, the quadratic Equation \eqref{eq:Quadratic_Equation} for $2a^2$ must have a non-negative solution. Given that for $t \gg 1$ we have $-\BBB<0$ according to Lemma \ref{lem:Coefficient_Large_Radius}, we must therefore have $\sqrt{\BBB^2 - 4 \AAA \CCC} > \BBB$ which is the case if and only if $\AAA \CCC < 0$. Given that $\AAA>0$ for large $t$, we must have $\CCC< 0$. This condition is, due to Equation \eqref{eq:signCvsslopecomparison}, equivalent to asking that
	\[\mu(L_{\rho_1}) < \mu(V_\rho),\]
	or equivalently
	\[\mu(L_{\rho_2}) > \mu(V_\rho).\]
	This condition is satisfied for any K\"ahler forms $\omega$ whose cohomology class lies in the non-empty cone $\mathcal{K}_2$ obtained in Corollary \ref{cor:Cones}. Suppose that $\omega$ is such a class. For sufficiently large $t$ we have that there is a positive solution
	\[2a^2 = \frac{-\BBB + \sqrt{\BBB^2 - 4 \AAA \CCC}}{2 \AAA }.\]
	Let 
	\begin{align*}
		\AAA_0 & := \frac{B_1^2 B_2^2 B_3^2}{3} , 
		&\BBB_0 & := B_1^2 B_2^4 B_3^4 , 
		&\CCC_0 & := B_1^2 B_2^2 B_3^2 \sum_{i=1}^3 \epsilon_i m_i B_j^2 B_k^2,
	\end{align*}
	so that 
	\begin{align*}
		\AAA & = t^6 \AAA_0 + O(t^2) , 
		&\BBB & =  t^{10} \BBB_0 +O(t^6) , 
		&\CCC  &=  t^{10} \CCC_0  + O(t^6) .
	\end{align*}
	In particular, we find that
	\[2a^2 = t^4 \left( \frac{-\BBB_0 + \sqrt{\BBB_0^2 - 4 \AAA_0 \CCC_0 t^{-4} }}{2 \AAA_0 } \right) + O(t^3)>0,\]
	 In particular, this gives an irreducible solution the dHYM equation.
\end{proof}

\subsection{The small radius regime}

Finally we turn to the main contribution of this article. In Lemma \ref{lem:Coefficient_Small_Radius}, we compute the coefficients of the quadratic Equation \eqref{eq:Quadratic_Equation} in the small radius regime and use these to prove Theorem \ref{thm:Small_Radius} which gives the existence of irreducible dHYM connections in the small radius regime.

\begin{lem}\label{lem:Coefficient_Small_Radius}
	Let $B_i>0$ be such that $B_2^2=B_1^2+B_3^2$ (so to guarantee that the Hermitian structure is K\"ahler) and consider $A_i=tB_i$ for $i \in \{ 1,2,3\}$ with $t \ll 1$. Then, as $t\to0$, we have
	\begin{align}
		\AAA & =  t^2 \frac{\sumcyclic \epsilon_i B_i^2 (p_jp_k +m_jm_k)}3 + O(t^6)  , \label{eq:Asmall}\\
		\BBB & = t^2 \biggl( \bigl( m_2m_3+p_2p_3 + \frac{m_1}{3} \bigr) \sumcyclic \epsilon_i B_i^2 (p_jp_k +m_jm_k)  \biggr.\notag\\* &\phantom{=t^2 \biggl(-}    - \bigl( p_1 p_2 p_3 + \sumcyclic p_i m_j m_k \bigr)( p_2\epsilon_3 B_3^2 + p_3\epsilon_2 B_2^2 )  \biggr) +O(t^6) , \label{eq:Bsmall}\\
		\CCC & = t^{2}  \biggl( m_1 m_2 m_3 + \sumcyclic m_i p_jp_k \biggr) \sumcyclic \epsilon_i B_i^2 (p_jp_k +m_jm_k)\notag \\*
		& \quad- t^2 \biggl( p_1p_2p_3 +\sumcyclic p_im_jm_k \biggr) \sumcyclic  \epsilon_i B_i^2 ( p_j m_k + m_j p_k )  
	+ O(t^6).\label{eq:Csmall}
	\end{align}
\end{lem}
\begin{proof}
	The proof follows from simply replacing $A_i=tB_i$, for $i \in \{ i,2,3\}$, into the formulas for $\AAA$, $\BBB$, and $\CCC$ given by Equations \eqref{eq:defA}, \eqref{eq:defB}, and \eqref{eq:defC}. 
\end{proof}

\begin{theorem}\label{thm:Small_Radius}
		Let $\rho=(-ir_1,0)$ and consider the complex vector bundle $V_\rho=L_{\rho_1} \oplus \underline{\bbC}$. Then, for any K\"ahler class $[\omega] \in \Homology^2(\bbF, \Reals)$, there is $\tau>0$ such that for all $t<\tau$, there is an irreducible dHYM connection on $V_\rho$ with respect to the K\"ahler form $t\omega$ . 
\end{theorem}

\begin{proof}In this context, $\beta_m=\rho_1-\rho_2=-\ii r_1=\beta_{(-2,1)}$, while $\beta_p=\rho_1+\rho_2=\beta_m$. So $(p_1,p_3)=(m_1,m_3)=(-2,1)$, and $p_2=m_2=1$. Substituting these values in Equations \eqref{eq:Asmall}, \eqref{eq:Bsmall}, \eqref{eq:Csmall}, we find that as $t\to 0$ we have
	\begin{align*}
		\AAA&= \frac23(\epsilon_1B_1^2-2\epsilon_2B_2^2-2\epsilon_3B_3^2)t^2+O(t^6),\\
		\BBB&= \frac{8}3(\epsilon_1B_1^2+\epsilon_2B_2^2+\epsilon_3B_3^2)t^2+O(t^6),\\
		\CCC&=O(t^6).\\
	\end{align*}
To pick up the next term in the asymptotic of $\CCC$, we impose $p=m$ in Equation \eqref{eq:defC}, and obtain
\begin{equation*}
	\CCC|_{p=m}=\left( 4m_1 m_2 m_3 + t^4\sumcyclic  \epsilon_i m_iB_j^2 B_k^2 \right)  B_1^2 B_2^2 B_3^2t^6. 
\end{equation*}
Substituting the actual values for $m$, we then obtain
$\CCC=-8B_1^2 B_2^2 B_3^2t^6+O(t^{10})$.
Substituting $(\epsilon_1,\epsilon_2,\epsilon_3)=(1,-1,1)$, and using $B_2^2=B_1^2+B_3^2$, we find  that $\BBB=O(t^6)$, so we also need to pick up the next term in the asymptotic. After doing so, we find
\begin{align*}
	\AAA&= 2B_1^2 t^2+O(t^6),\\
	\BBB&=\bigl(B_1^6+\frac{16}3B_1^2B_2^2B_3^2\bigr)t^6+O(t^{10}),\\
	\CCC&=-8B_1^2 B_2^2 B_3^2t^6+O(t^{10}).
\end{align*}

	In particular, we find that for small enough $t>0$, $\AAA$ is positive and $\CCC$ is negative, and therefore $-\AAA\CCC>0$.  The quadratic Equation \eqref{eq:Quadratic_Equation} thus always has a positive root
	\[2a^2 = \frac{-\BBB + \sqrt{\BBB^2 - 4 \AAA \CCC }}{2 \AAA}.\]
	The proof is now complete.
\end{proof}

\begin{remark} Since $\BBB$ is positive for small $t$, the the other root of the quadratic equation is negative and therefore it cannot be equal to $2a^2$ for some real $a$.
\end{remark}

\appendix

\section{An important necessary condition}

\begin{thm}
	Let $\kappa$ be a solution to the dHYM equation with angle $\theta \in \left(- \frac{n\pi}{2} ,  \frac{n\pi}{2} \right)$. Then, for any hypersurface $\surface\subset X$
	\[\int_\surface e^{-\ii \left(\theta- \frac{\pi}{2}\right)} (\omega + \ii \kappa)^{n-1} >0 .\]
	In other words
	\[\Im \Bigl( \frac{Z_{\surface}(\kappa)}{Z_{\bbF_2}(\kappa)}\Bigr) >0 .\]
\end{thm}
\begin{proof} 
Let $\{\alpha_1,\ldots,\alpha_n\}$ be a Darboux co-framing diagonalizing both $\omega$ and $\kappa$, so that 
	\[\omega =  \frac{\ii}{2}\sum_{i=1}^n \alpha_i \wedge \overline{\alpha_i} , \ \ \text{and} \ \ \kappa = \frac{\ii}{2}\sum_{i=1}^n \lambda_i  \alpha_i \wedge \overline{\alpha_i}. \]
	Recall that $\theta=\sum_{i=1}^n \arctan(\lambda_j)$ and therefore, for any $j \in \lbrace 1, \ldots , n \rbrace$,
	\[\theta - \frac{\pi}{2} < \sum_{i \neq j} \arctan (\lambda_i ) < \theta + \frac{\pi}{2} ,\]
 thus
	\begin{equation}\label{eq:sum_of_arctan_minus_one}
		0 <  \sum_{i \neq j} \arctan(\lambda_i) - \bigl(\theta- \frac{\pi}{2}\bigr) < \pi.
	\end{equation}
	
	Then, $\omega + i \kappa =\frac{\ii}{2} \sum_{i=1}^n ( 1+\ii\lambda_i )  \alpha_i \wedge \overline{\alpha_i} $ and therefore
	\begin{align*}
		(\omega +\ii \kappa)^{n-1} = \sum_{j=1}^n \left( \prod_{i \neq j} \sqrt{1+\lambda_i^2} \right) \exp \left(\ii \sum_{i \neq j} \arctan(\lambda_i)\right) \left( \prod_{i \neq j} \frac{\ii}{2} \alpha_j \wedge \overline{\alpha_j} \right).
	\end{align*}
	Furthermore, for any complex hyperplane $v \subset (T_{\bbC}X)_p$ we have $c_j(v) = \prod_{i \neq j} \frac{\ii}{2} \alpha_j \wedge \overline{\alpha_j}  (v) \in \Reals$ is nonnegative, with at least one being positive, and therefore
	\begin{align*}
		\Im \left( e^{-\ii (\theta- \frac{\pi}{2})} (\omega + \ii \kappa)^{n-1} (v) \right) & = \sum_{j=1}^n  c_j (v) \prod_{i \neq j} \sqrt{1+\lambda_i^2}  \Im \exp \left(\ii \left( \sum_{i \neq j} \arctan(\lambda_i) - (\theta- \frac{\pi}{2}) \right) \right).
	\end{align*}
	Combining this with Equation \eqref{eq:sum_of_arctan_minus_one}, we conclude that
	\begin{align*}
		\Im \left( e^{-\ii (\theta- \frac{\pi}{2})} (\omega + \ii \kappa)^{n-1} (v) \right) & >0 ,
	\end{align*}
	which upon integration over $\surface$ yields the desired result.
\end{proof}

\printbibliography
\end{document}

\typeout{get arXiv to do 4 passes: Label(s) may have changed. Rerun}